\documentclass[twoside,11pt]{amsart}

\usepackage{graphics}
\usepackage{graphicx}
\usepackage{epsfig}
\usepackage{amsmath}
\usepackage{amssymb}
\usepackage{fancybox}
\usepackage{amsfonts,url,times}
\usepackage{amscd}

\newtheorem{theorem}{Theorem}[section] 
\newtheorem{lemma}[theorem]{Lemma}     
\newtheorem{corollary}[theorem]{Corollary}
\newtheorem{proposition}[theorem]{Proposition}

\newcommand{\p}{\mathfrak{p}}
\newcommand{\q}{\mathfrak{q}}

\newcommand{\rr}{\mathfrak{r}}

\newcommand{\Z}{\mathbb{Z}}
\newcommand{\C}{\mathbb{C}}

\newcommand{\Q}{\mathbb{Q}}

\newcommand{\A}{\mathbb{A}}

\newcommand{\ff}{\mathfrak{f}}

\title[Non-abelian congruences between special values of $L$-functions of elliptic curves; the CM case]
 {Non-abelian congruences between special values of $L$-functions of elliptic curves; the CM case} 

\author{Thanasis Bouganis}
\address{Universit\"{a}t Heidelberg\\ Mathematisches Institut\\
Im Neuenheimer Feld 288\\ 69120 Heidelberg, Germany.}
\email{bouganis@mathi.uni-heidelberg.de}

\begin{document}
\maketitle

\begin{abstract}
In this work we prove congruences between special values of elliptic
curves with CM that seem to play a central role in the analytic side
of the non-commutative Iwasawa theory. These congruences are the
analogue for elliptic curves with CM of those proved by Kato, Ritter
and Weiss for the Tate motive. The proof is based on the fact that
the critical values of elliptic curves with CM, or what amounts to
the same, the critical values of Gr\"{o}ssencharacters, can be
expressed as values of Hilbert-Eisenstein series at CM points. We
believe that our strategy can be generalized to provide congruences
for a large class of $L$-values.
\newline

\end{abstract}

\section{Introduction}

In \cite{CFKSV,FK} a vast generalization of the Main Conjecture of
the classical (abelian) Iwasawa theory to a non-abelian setting is
proposed. As in the classical theory, the non-abelian Main
Conjecture predicts a deep relation between an analytic object (a
non-abelian $p$-adic $L$-function) and an algebraic object (a Selmer
group or complex over a non-abelian $p$-adic Lie extension). However
the evidences for this non-abelian Main Conjecture are still very
modest. One of the central difficulties of the theory seems to be
the construction of non-abelian $p$-adic $L$-functions. Actually the
only known results in this direction are mainly restricted to the
Tate motive over particular $p$-adic Lie extensions as for example
in \cite{RW,Kato,Kakde,Hara}. We should also mention here that for
elliptic curves there are some evidences for the existence of such
non-abelian $p$-adic $L$-functions offered in
\cite{Bouganis2,DelbWard} and also some computational evidences
offered in \cite{Dokchitsers,DelbWard2}.

The main aim of the present work is to address the issue of the
existence of the non-abelian $p$-adic $L$-function for an elliptic
curve with complex multiplication (but see also the remark later in
the introduction) with respect specific $p$-adic Lie extension as
for example, the so-called false Tate curve extension or Heisenberg
type Lie extensions. Actually we prove congruences, under some
assumptions, that are the analogue for elliptic curves with CM of
those proved by Ritter and Weiss in \cite{RW} for the Tate motive.
We remark that such congruences can be used to prove the existence
of the non-abelian $p$-adic $L$-function as done for example in
\cite{Kato} or in \cite{Kakde} for the Tate motive. We start by
making our setting concrete.

Let $E$ be an elliptic curve defined over $\Q$ with CM by the ring
of integers $\mathfrak{R}_0$ of a quadratic imaginary field $K_0$.
We fix an isomorphism $\mathfrak{R}_0 \cong End(E)$ and we write
$\Sigma_0$ for the implicit $CM$ type. Let us write $\psi_{K_0}$ for
the attached Gr\"{o}ssencharacter to $E$, that is $\psi_{K_0}$ is a
Hecke character of $K$ of (ideal) type $(1,0)$ with respect to the
CM type $\Sigma_0$ and satisfy $L(E,s)=L(\psi_{K_0},s)$. We fix an
odd prime $p$ where the elliptic curve has good ordinary reduction.
We fix an embedding $\bar{\Q}\hookrightarrow \bar{\Q}_p$ and using
the selected CM type we fix an embedding $K \hookrightarrow
\bar{\Q}$. The ordinary assumption implies that $p$ splits in $K_0$,
say to $\p$ and $\bar{\p}$ where we write $\p$ for the prime ideal
that corresponds to the $p$-adic embedding $K \hookrightarrow
\bar{\Q}\hookrightarrow \bar{\Q}_p$. We write $N_E$ for the
conductor of $E$ and $\mathfrak{f}$ for the conductor of
$\psi_{K_0}$.

We consider a finite totally real extension $F$ of $\Q$ which we
assume unramified at the primes of $\Q$ that ramify in $K_0$ and at
$p$. We write $\mathfrak{r}$ for its ring of integers and we fix an
integral ideal $\mathfrak{n}$ of $\mathfrak{r}$ that is relative
prime to $p$ and to $N_E$. Let now $F'$ be a totally real Galois
extension of $F$, cyclic of order $p$ that is ramified only at
primes of $F$ lying above $p$ or at primes of $F$ that divide
$\mathfrak{n}$. We make the additional assumptions that the non-$p$
part of the conductor of $F'/F$ divides $\mathfrak{n}$, that is $F'$
is a subfield of $F(p^{\infty}\mathfrak{n})$, the ray class field of
conductor $p^\infty\mathfrak{n}$ and that the primes that ramify in
$F'/F$ are split in $K$. That is if we write $\theta_{F'/F}$ for the
relative different of $F'/F$ then
$\theta_{F'/F}=\mathfrak{P}\bar{\mathfrak{P}}$ in $K'$. We write
$\mathfrak{r'}$ for the ring of integers in $F'$. We write $K$ for
the CM field $F K_0$ and $K'$ for the CM field $F' K_0=F'K$ and
$\mathfrak{R}$ and $\mathfrak{R'}$ for their ring of integers
respectively. We also write $\Gamma$ for the Galois group $Gal(F'/F)
\cong Gal(K'/K)$. Note that in both $F$ and $F'$ all primes above
$p$ split in $K$ and $K'$ respectively. Finally we write $\tau$ for
the nontrivial element (complex conjugation) of $Gal(K/F) \cong
Gal(K'/F')$ and we set $g:=[F:\Q]$.

We now consider the base changed elliptic curves $E/F$ over $F$ and
$E/F'$ over $F'$. We note that the above setting gives the following
equalities between the $L$ functions,
\[
L(E/F,s) = L(\psi_K,s),\,\,\,L(E/F',s)=L(\psi_{K'},s)
\]
where $\psi_K := \psi_{K_0} \circ N_{K/K_0}$ and $\psi_{K'}:= \psi_K
\circ N_{K'/K} = \psi_{K_0} \circ N_{K'/K_0}$, that is the
base-changed characters of $\psi_{K_0}$ to $K$ and $K'$.

We write $G_F$ for the Galois group
$Gal(F(p^{\infty}\mathfrak{n})/F)$ where $F(p^{\infty}\mathfrak{n})$
denotes the ray class field modulo $p^\infty\mathfrak{n}$ over $F$,
and also $G_{F'}:=Gal(F'(p^{\infty}\mathfrak{n})/F')$ for the
analogue for $F'$. Note that the above setting introduce a transfer
map $ver:G_F \rightarrow G_{F'}$. Moreover we have an action by
conjugation of $\Gamma=Gal(F'/F)$ on $G_{F'}$. We consider the
measures $\mu_{E/F}$ of $G_F$ and $\mu_{E/F'}$ of $G_{F'}$ that
interpolate the critical values of the elliptic curve $E/F$ and
$E/F'$ respectively twisted by finite order characters of conductor
dividing $p^\infty\mathfrak{n}$. The precise interpolation
properties is a delicate issue in our setting that we will discuss
in the next section. However we can state now the main theorem of
our work. We write $\mathfrak{j}$ for the smallest ideal of
$\mathfrak{r}$ which contains $\mathfrak{n}\mathfrak{f}\mathfrak{R}
\cap F$ and its prime factors inert or ramify in $K$. If we write
$\mathfrak{J}:=\mathfrak{j}\mathfrak{R}$ then we denote by
$Cl_{K}(\mathfrak{J})$ the ray class group of the ray class field
$K(\mathfrak{J})$. We define $Cl^{-}_{K}(\mathfrak{J})$ as the
quotient of $Cl_K({\mathfrak{J}})$ by the natural image of
$(\mathfrak{r}/\mathfrak{j})^\times$. Similarly we make the
analogous definitions for $K'$.

\begin{theorem}\label{maintheorem}We make the assumptions
\begin{enumerate}
\item  The natural map $Cl^{-}_K(\mathfrak{J}) \rightarrow Cl^{-}_{K'}(\mathfrak{J})^{\Gamma}$ is an isomorphism,
\item The natural map $Cl_{F}(1) \rightarrow Cl_{F'}(1)$ is an injection,
\item The relative different $\theta_{F'/F}$ of $F'$ over $F$ is trivial in $Cl^{+}_{F'}$, the strict ideal class group of
$F'$.  That is, there is $\xi \in F'$, totally positive so that
$\theta_{F'/F} = (\xi)$.
\end{enumerate}
Then,
\begin{equation}
\int_{G_F}\epsilon \circ ver \,\,\,\, d\mu_{E/F} \equiv
\int_{G_{F'}}\epsilon\,\,\,\,\, d\mu_{E/F'} \mod{p\Z_p}
\end{equation}
for all $\epsilon$ locally constant $\Z_p$-valued functions on
$G_{F'}$ such that $\epsilon^{\gamma}=\epsilon$ for all $\gamma \in
\Gamma$ where
$\epsilon^\gamma(g):=\epsilon(\tilde{\gamma}g\tilde{\gamma}^{-1})$
for all $g \in G_{F'}$ and for some lift $\tilde{\gamma} \in
Gal(F'(p^\infty\mathfrak{n})/F))$ of $\gamma$. More generally if we
relax the assumption (1) and assume only that
$\imath:Cl^{-}_K(\mathfrak{J}) \hookrightarrow
Cl^{-}_{K'}(\mathfrak{J})^{\Gamma}$ is injective then equation (1)
reads
\begin{equation}
\int_{G_F}\epsilon \circ ver \,\,\,\, d\mu_{E/F} \equiv
\int_{G_{F'}}\epsilon\,\,\,\,\, d\mu_{E/F'} + \Delta(\epsilon)
\mod{p\Z_p}
\end{equation}
where $\Delta(\epsilon)$ is an ``error term'' that depends on the
cokernel of the map $\imath$.
\end{theorem}

{\textbf{Remarks}:}
\begin{enumerate}
\item It can be shown, see for example \cite{RW}, that the above congruences imply the
following congruences between measures. If we write $f_{E/F}$ for
the element in the Iwasawa algebra $\Z_p[[G_F]]$ that corresponds to
the measure $\mu_{E/F}$ and similarly $f_{E/F'}$ for that in
$\Z_p[[G_F']]$ then we obtain the congruences
\[
ver(f_{E/F}) \equiv f_{E/F'} \mod\,\,T
\]
where $T$ is the trace ideal in $\Z_p[[G_F']]^{\Gamma}$ generated by
elements $\Sigma_{i=0}^{p-1} a^{\gamma^{i}}$ for $a \in
\Z_p[[G_F']]$. Note that $f_{E/F'}$ is in $\Z_p[[G_F']]^{\Gamma}$ as
it comes from base change from $F$. It is exactly this implication
that motivates our work. The aim is to use this kind of congruences
to establish the existence of non-commutative $p$-adic $L$-functions
for our elliptic curve with respect to specific $p$-adic Lie groups,
as for example Heisenberg type Lie groups, very much in the same
spirit done by Kato for the Tate motive $\mathbb{Z}_p(1)$ in
\cite{Kato} and by Kakde for false Tate curve extensions also for
the Tate motive in \cite{Kakde}.

\item Our assumption that the elliptic curve is defined over
$\mathbb{Q}$ is made mainly for simplicity reasons. Our
considerations could be applied in a more general setting. One can
consider as starting ``object'' a Hilbert-modular form over $F$ with
CM by $K$. The delicate issue however is the understanding of the
various motivic periods that are associated to it. However the
``philosophy'' of our proof applies also in this setting.

\item We believe that the term $\Delta(\epsilon)$ always vanishes
but we cannot prove it yet.

\item The assumption that $\epsilon$ is $\mathbb{Z}_p$-valued can be
relaxed and consider any integrally-valued locally constant
function. Then simply one obtains the congruences
\[
\int_{G_F}\epsilon \circ ver \,\,\,\, d\mu_{E/F} \equiv
\int_{G_{F'}}\epsilon\,\,\,\,\, d\mu_{E/F'} \mod{p\Z_p[\epsilon]}
\]
where $\Z_p[\epsilon]$ is the ring of integers of the smallest
extension of $\Q_p$ that contains the values of $\epsilon$.

\end{enumerate}

{\textbf{On the strategy of the proof:}} Let us finish the
introduction by explaining briefly the main idea of the proof. As we
will explain shortly we are going to construct the measure
$\mu_{E/F}$ and $\mu_{E/F'}$ by using the so-called Katz measure for
Hecke characters of CM fields. The reason for this should be
intuitively clear from the equation of $L$ functions above. These
measures are constructed by using the fact (going back to Damerell's
theorem) that the special values of the $L$ of Gr\"{o}ssencharacters
can be expressed as values of Hilbert-Eisenstein series on CM
points. The modular meaning of these CM points is that they
correspond to Hilbert-Blumenthal abelian varieties (HBAV) with CM of
the same type as the character under consideration. In our relative
setting we have that the Gr\"{o}ssencharacter $\psi_{K'}$ is the
base change of $\psi_K$, in particular as we will explain in the
next section if we write $(K,\Sigma)$ for the CM type of $\psi_{K}$
then the CM type of $\psi_{K'}$ is $(K',\Sigma')$ where $\Sigma'$ is
the lift of $\Sigma$ to $K'$. But now the key observation is that
the HBAV with CM of type $(K',\Sigma')$ are isogenous to
$[K':K]$-copies of HBAV with CM $(K,\Sigma)$. In particular this
says that the CM points that we need to evaluate our Eisenstein
series over $F'$ are in some sense coming from $F$ through the
natural diagonal embedding $\Delta:\mathbb{H}_F \hookrightarrow
\mathbb{H}_{F'}$ of the Hilbert upper half planes. Note here the
importance of $\Sigma'$ being lifted from $\Sigma$. Hence we can
pull-back the Hilbert-Eisenstein that is used over $F'$ to obtain a
Hilbert-Eisenstein series over $F$, so that its values on CM points
of $\mathbb{H}_F$ are the same with those of the one over $F'$
evaluated on the image of the CM points with respect to $\Delta$. It
is mainly this idea that we will use to prove the above congruences.
We note here that a similar strategy was used by Kato \cite{Kato}
and Ritter and Weiss \cite{RW} for the cyclotomic character but in
their works the $L$ values appeared as the constant term of
Hilbert-Eisenstein series (or as ``values'' at the cusp at
infinity). We believe that this strategy is more general. We have
applied similar considerations in \cite{B} for other $L$-values that
can be understood either as values at CM points or at infinity of
Eisenstein series of the unitary group. Actually what we are doing
here could be rephrased in the unitary group setting, but we defer
this discussion for our forthcoming work \cite{B}.

\section{The measures attached to the elliptic curves $E/F$ and
$E/F'$}

The statement of our main theorem involves measures on $G_F$ (resp
$G_{F'}$) with the property that integrating these measures against
finite characters of $G_F$ (resp $G_{F'}$) we obtain critical values
of $E/F$ (resp $E/F')$) twisted with these characters up to some
modification. Now we proceed in explaining the construction of these
measures and their interpolation properties. We point right away
that there are various construction of these measures; the modular
symbol construction which we will not discuss at all, the
construction of Katz, Hida and Tilouine on which we will use in the
present work and finally in our specific setting the construction of
Colmez and Schneps which we also discuss shortly below. In order to
explain the definition of the above-mentioned measures we need to
introduce some more notation.

{\textbf{Archimedean and $p$-adic periods:}} Since the elliptic
curve $E$ is defined of $\Q$, we have that the class number of $K_0$
is one. In particular we can fix a well-defined complex period for
$E$ as follows. We write $\Lambda$ for the lattice of $E$, that is
$E(\C) \cong \C/\Lambda$. Then we define $\Omega_{\infty}(E) \in
\C^{\times}$ uniquely up to elements in $\mathfrak{R}^{\times}_0$ as
$\Lambda = \Omega_{\infty}(E) \sigma_0(\mathfrak{R}_0)$, where
$\sigma_0 :K \hookrightarrow \C$ the selected embedding. Moreover we
define a $p$-adic period $\Omega_{p}(E) \in
\mathcal{J}_{\infty}^{\times}$, where $\mathcal{J}_{\infty}$ denotes
the ring of integers of the $p$-adic completion of the maximal
abelian unramified extension of $\Q_p$. If we write $\Phi$ for the
extension of Frobenious that operates on it, then it is well-known
that this period is uniquely determined up to elements in
$\Z_p^{\times}$ by the property
\[
\frac{\Omega_{p}(E)^{\Phi}}{\Omega_{p}(E)} = u \in \Z_{p}^{\times}
\]
where $u$ is the $p$-adic number determined by the equation (as $p$
is good ordinary for $E$)
\[
1-a_p(E)X+pX^{2}=(1-uX)(1-wX)
\]

{\textbf{CM-types:}} We fix some CM-types for the CM fields
$K_0,K,K'$. We have fixed already an embedding of $K_0
\hookrightarrow \C$, say $\sigma_0$ and defined the CM type of $K_0$
by $\Sigma_0 = \{\sigma_0\}$. We normalized things so that the
character $\psi_{K_0}$ is of infinite type $1\,\,\Sigma_0$. Now we
fix a CM type $\Sigma$ of $K$ by taking the lift of $\Sigma_0$ to
$K$. That is, we pick embeddings that restrict to $\sigma_0$ in
$K_0$. We also define $\Sigma'$ to be the lift of $\Sigma$ in $K'$.
We note two things for these CM-types. First the characters $\psi_K$
and $\psi_K'$ are of type $1 \Sigma$ and $1 \Sigma'$. Second the
types just picked are also $p$-ordinary in the terminology of Katz,
that simply amounts of picking the primes of $K$ and $K'$ that are
above $\p$ and not $\bar{\p}$. We denote these sets of primes as
$\Sigma_p$ and $\Sigma'_p$ respectively. Of course we set also
$\Sigma_{0,p}=\{\p\}$. Finally we note that all abelian varieties of
dimension $[F:\Q]$ with CM by $K$ (resp dimension $[F':\Q]$ and CM
by $K'$) and type $\Sigma$ (resp. $\Sigma'$) are isogenous to the
product of $[F:\Q]$ (resp. $[F':\Q]$) copies of the elliptic curve
$E$.

{\textbf{The $\infty$-types of the Gr\"{o}ssencharacters:}} For the
Gr\"{o}ssencharacter $\psi_{K_0}$ we have that
$\psi_{K_0}\bar{\psi}_{K_0}=N_{K/\Q}$ and that
$\psi_{K_0}(\bar{\mathfrak{q}})=\overline{{\psi}_{K_0}(\mathfrak{q})}$.
In particular we have that
\[
L(\psi_{K_0}^{-1},0)=L(\bar{\psi}_{K_0}N^{-1}_{K/\Q},0)=L(\bar{\psi}_{K_0},1)=L(\psi_{K_0},1)
\]
Moreover if we consider twists by finite cyclotomic characters, that
is characters of the form $\chi=\chi' \circ N_{K_{0}/\Q}$ for
$\chi'$ a finite Dirichlet character of $\Q$, we have that
$L(\psi_{K_0}\chi,1)=L(\psi_{K_0}^{-1}\chi,0)$. So from now on we
are going to consider characters of infinite type $-k\Sigma_0$,
$-k\Sigma$ and $-k\Sigma'$ for the various CM-types and $k\geq 1$
and the critical values that we study are at $s=0$. The above
equation explains why these are the values that we are interested
in.

We now recall the interpolation properties of a slight modification
of a $p$-adic measure $\mu_{\delta}^{KHT}$ for Hecke characters
constructed by Katz \cite{Katz2} and later extended by Hida and
Tilouine in \cite{HT}. Let $\mathfrak{C}$ be some integral ideal of
$K$ relative prime to $p$. Then for a Hecke character $\chi$ of
$G_K:=Gal(K(\mathfrak{C}p^{\infty})/K)$ of infinite type $-k\Sigma$
we have
\[
\frac{\int_{G_K}\chi(g)\mu_{\delta}^{KHT}(g)}{\Omega_p^{k\Sigma}} =
(\mathfrak{R}^{\times}:\rr^{\times})Local(\Sigma,\chi,\delta)\frac{(-1)^{kg}\Gamma(k)^{g}}{\sqrt{D_F}\Omega_{\infty}^{k\Sigma}}\times
\]
\[
\prod_{\mathfrak{q}|\mathfrak{FJ}}(1-\check{\chi}(\bar{\mathfrak{q}}))\prod_{\
\q | \mathfrak{F}}(1-\chi(\bar{\q}))\prod_{\mathfrak{p}\in
\Sigma_p}(1-\chi(\bar{\mathfrak{p}}))(1-\check{\chi}(\bar{\mathfrak{p}}))L(0,\chi)
\]
where the ideals $\mathfrak{F},\mathfrak{J}$ are some factors of
$\mathfrak{C}$ and will be defined in the next section. Also in the
next section we will explain in details the construction of the
above measure but for the time being we just want to indicate three
points:
\begin{enumerate}
\item The measure depends on a choice of an element $\delta \in K$,
totally imaginary with respect to $\Sigma$ and such that its
valuation at $\p \in \Sigma_p$ is equal with the valuation of the
absolute different of $K$.
\item The periods (archimedean and $p$-adic) that appear above depend only on the CM type $\Sigma$
and not at all on the finite part of the Hecke character $\chi$.
\item The factor $Local(\chi,\Sigma,\delta)$ is similar to some
epsilon factor of $\chi$ but not equal. We will explain more on that
shortly.
\end{enumerate}
We have fixed above a Gr\"{o}ssencharacter $\psi_K$ (note that $k=1$
for this character). We set, with notation as in the introduction,
$\mathfrak{C}:=\mathfrak{n}\mathfrak{f}\mathfrak{R}$ and we consider
the measure of $G_K$ defined for every finite character $\chi$ of
$G_K$ by
\[
\int_{G_K}\chi(g)\mu_{\psi_{K},\,\delta}^{KHT}(g):=\int_{G_K}\chi(g)\hat{\psi}^{-1}_K(g)\mu_{\delta}^{KHT}(g)
\]
where $\hat{\psi}_K$ is the $p$-adic avatar of $\psi_K$ constructed
by Weil. We will show later that in this case we can set
$\Omega_p^{\Sigma}= \Omega_p(E)^g$ and
$\Omega_{\infty}^{\Sigma}=\Omega(E)^g$. Then we define the measure
$\mu_{E/F}$ discussed above by (recall that
$G_F:=Gal(F(\mathfrak{n}p^\infty)/F)$) by
\[
\int_{G_F}\chi(g)
\mu_{E/F}(g):=\frac{\int_{G_K}\tilde{\chi}(g)\mu_{\psi_{K},\,\delta}^{KHT}(g)}{\Omega_p(E)^g}
\]
where $\tilde{\chi}$ is the base change of $\chi$ from $F$ to $K$.
Then from our remarks on the critical value $L(E/F,1)$ we see that
this measure interpolates twists of this critical value of $E/F$.
The same considerations apply also for the datum $(K',F',
\psi_{K'},G_{F'},G_{K'})$. We now observe that our main theorem
above amounts to prove the following congruences, under of course
the same assumptions as in the theorem above,
\[
\frac{\int_{G_K}\epsilon \circ ver \,\,\,\,
d\mu^{KHT}_{\psi_K,\,\delta}}{\Omega_{p}(E)^g} \equiv
\frac{\int_{G_{K'}}\epsilon\,\,\,\,\,
d\mu^{KHT}_{\psi_{K'},\,\delta'}}{\Omega_p(E)^{pg}} \mod{p\Z_p}
\]
for all $\epsilon$ locally constant $\Z_p$-valued functions on
$G_{K'}$ with $\epsilon^{\gamma}=\epsilon$ and belong to the
cyclotomic part of it, i.e. when it is written as a sum of finite
order characters it is of the form $\epsilon= \sum c_{\chi}\chi$
with $\chi^{\tau} = \chi$.

However these congruences do not hold when the extension $F'/F$ is
ramified. In order to overcome this difficulty we will need to
modify (twist) the measure of Katz-Hida-Tilouine over $K'$. The key
question is whether the factor $Local(\chi,\Sigma,\delta)$ is the
``right'' one. We believe that this is not so when the extension
$F'/F$ is ramified (in the appendix we offer evidences for this) and
actually with our modification we aim to overcome this problem. In
short, we will define for the datum $(K',F',
\psi_{K'},G_{F'},G_{K'})$ the measure $\mu_{E/F'}$ as
\[
\int_{G_{F'}}\chi(g)
\mu_{E/F'}(g):=\frac{\int_{G_{K'}}\tilde{\chi}(g)\mu_{\psi_{K'},\,\delta,
\xi}^{KHT,tw}(g)}{\Omega_p(E)^{pg}}:=\frac{\int_{G_{K'}}\chi(g)\hat{\psi}^{-1}_{K'}(g)\mu_{\delta,
\xi}^{KHT,tw}(g)}{\Omega_p(E)^{pg}}
\]
where the measure $\mu_{\delta, \xi}^{KHT,tw}$, called in this work
the twisted Katz-Hida-Tilouine measure, will be defined in section
4. Then we will show that
\[
\frac{\int_{G_K}\epsilon \circ ver \,\,\,\,
d\mu^{KHT}_{\psi_K,\,\delta}}{\Omega_{p}(E)^g} \equiv
\frac{\int_{G_{K'}}\epsilon\,\,\,\,\,
d\mu^{KHT,tw}_{\psi_{K'},\,\delta,\xi}}{\Omega_p(E)^{pg}}
\mod{p\Z_p}
\]
for all $\epsilon$ locally constant $\Z_p$-valued functions on
$G_{K'}$ with $\epsilon^{\gamma}=\epsilon$ and belong to the
cyclotomic part of it.

{\textbf{The measure of Colmez and Schneps:}} We close this section
by making a few more observations. In the setting that we consider
we can apply the construction of \cite{CS}. Indeed in this work
Colmez and Schneps construct a measure of
$G_{K}:=Gal(K(\mathfrak{C}p^\infty)/K)$ such that for every
Gr\"{o}ssencharacter $\chi$ of $K$ of infinite type $\chi((a))=
N_{K/K_0}(a))^{-k}$ for $a \equiv 1$ modulo the conductor of $\chi$
has the interpolation property
\[
\int_{G}\chi(g)\mu^{CS}(g) =
(-1)^{kg}\Gamma(k)^{g}\prod_{\mathfrak{p} \in
\Sigma_p}e_{\mathfrak{p}}(\chi,\psi,dx_1)
\prod_{\mathfrak{q}|\mathfrak{C}}(1-\chi(\mathfrak{q}))\prod_{\mathfrak{p}\in
\Sigma_p}(1-\chi(\bar{\mathfrak{p}}))(1-\check{\chi}(\bar{\mathfrak{p}}))L(0,\chi)
\]
Although Colmez and Schneps do not work the algebraicity of the
measure we see here that their measure is normalized differently
from that of Katz-Hida-Tilouine with respect to the local factors.
Here one gets the epsilon factors of Deligne as local factors. It is
exactly this construction that we explore in a common work with
Filippo Nuccio \cite{BN} where we try to obtain a different proof of
the congruences hoping also to relax some of the assumptions of the
current work.

\section{The Eisenstein measure of Katz-Hida-Tilouine}

We start by recalling some Eisenstein series appearing in the work
of Katz \cite{Katz2} and Hida and Tilouine \cite{HT}. We follow the
notations of Hida and Tilouine and introduce the setting described
in their paper. We consider a totally real field $F$ with ring of
integers $\mathfrak{r}$ and write $\theta$ for the different of
$F/\Q$. We also fix an odd prime $p$. For an ideal $\mathfrak{a}$ of
$F$ we write $\mathfrak{a}^{*}=\mathfrak{a}^{-1}\theta^{-1}$. We fix
a fractional ideal $\mathfrak{c}$ and take two fractional ideals
$\mathfrak{a}$ and $\mathfrak{b}$ such that
$\mathfrak{a}\mathfrak{b}^{-1} = \mathfrak{c}$. Let $\phi :
\{\mathfrak{r}_p \times (\mathfrak{r}/\mathfrak{f}')\} \times
\{\mathfrak{r}_p \times (\mathfrak{r}/\mathfrak{f}'')\} \rightarrow
\C$ be a locally constant function such that
$\phi(\epsilon^{-1}x,\epsilon y)=N(\epsilon)^{k}\psi(x,y)$, for all
$\epsilon \in \mathfrak{r}^{\times}$, $k$ some positive integer and
$\mathfrak{f}'$ and $\mathfrak{f}''$ integral ideals relative prime
to $p$. We put $\mathfrak{f}:= \mathfrak{f}' \cap \mathfrak{f}''$
and we assume that the ideals $\mathfrak{a},\mathfrak{b}$ and
$\mathfrak{c}$ are prime to $\mathfrak{f}$. Moreover we assume that
the ideal $\mathfrak{a}$ is prime to $p$ and that $p$ does not
divide $\mathfrak{b}$. However we allow the case
$(p,\mathfrak{b}^{-1}) \neq 1$. We consider the natural projection
$T:=\{\mathfrak{r}_p \times (\mathfrak{r}/\mathfrak{f})\} \times
\{\mathfrak{r}_p \times (\mathfrak{r}/\mathfrak{f})\} \rightarrow
\{\mathfrak{r}_p \times (\mathfrak{r}/\mathfrak{f}')\} \times
\{\mathfrak{r}_p \times (\mathfrak{r}/\mathfrak{f}'')\}$ and
consider $\phi$ as a locally constant function on $T$.

We define the partial Fourier transform of the first variable of
$\phi$ and write
\[
P\phi : \{(F_p / \theta_p^{-1} \times \mathfrak{f}^*/\theta^{-1})
\times (\mathfrak{r}_p \times (\mathfrak{r}/\mathfrak{f}))\}
\rightarrow \C
\]
as
\[
P\phi(x,y)= p^{\alpha[F:\Q]N(\mathfrak{f})^{-1}} \sum_{a \in
X_{\alpha}}\phi(a,y)e_F(ax)
\]
for $\phi$ factoring through $X_\alpha \times \mathfrak{r}_p \times
(\mathfrak{r}/\mathfrak{f})$ with $X_\alpha:= \mathfrak{r}_p/\alpha
r_p \times (\mathfrak{r}/\mathfrak{f})$ with $\alpha \in
\mathbb{N}$.

We attach an Eisenstein series to $\phi$. This is realized as a rule
on triples $(\mathcal{L},\lambda,\imath)$ where $\imath$ a
$p^{\infty}\mathfrak{f}^2$ level structure.

{\textbf{The partial Tate module:}} From the $p^\infty
\mathfrak{f}^2$ structure after restriction we obtain a short exact
sequence of $\mathfrak{r}_p \times
\mathfrak{r}/\mathfrak{f}\mathfrak{r}$-modules
\[
0 \rightarrow \theta^{-1} \otimes (\mathfrak{r}_p \times
\mathfrak{r}/\mathfrak{f}\mathfrak{r}) \rightarrow \mathcal{L}
\otimes (\mathfrak{r}_p \times
\mathfrak{r}/\mathfrak{f}\mathfrak{r}) \rightarrow ? \rightarrow 0
\]
From the given polarization after we obtain an isomorphism
\[
\bigwedge^2_{\mathfrak{r}_p \times
\mathfrak{r}/\mathfrak{f}\mathfrak{r}}(\mathfrak{L}\otimes
(\mathfrak{r}_p \times \mathfrak{r}/\mathfrak{f}\mathfrak{r})) \cong
\theta^{-1}\mathfrak{c}^{-1} \otimes (\mathfrak{r}_p \times
\mathfrak{r}/\mathfrak{f}\mathfrak{r})
\]
From where we conclude that
\[
?\cong \mathfrak{c}^{-1} \otimes (\mathfrak{r}_p \times
\mathfrak{r}/\mathfrak{f}\mathfrak{r}) \cong \mathfrak{c}_p^{-1}
\mathfrak{r}_p \times \mathfrak{r}/\mathfrak{f}\mathfrak{r}
\]
We obtain the projection $\pi'$
\[ \pi' : (\mathcal{L} \otimes \rr_p) \times
\mathcal{L}/\ff\mathcal{L} \rightarrow \mathfrak{c}_p^{-1}
\mathfrak{r}_p \times \mathfrak{r}/\mathfrak{f}\mathfrak{r} \]

Following Hida and Tilouine we then define the partial Tate module
$PV(\mathcal{L})$ as a submodule of $\mathcal{L} \otimes F_{p\ff}$
that contains $\mathcal{L} \otimes \rr_{p\ff}$ such that
\[
PV(\mathcal{L})/\mathcal{L} \otimes \rr_{p\ff} \cong
Im(F_p/\theta^{-1} \times \ff^*/\theta^{-1} \rightarrow
p^{-\infty}\mathcal{L}/\mathcal{L} \times
\ff^{-1}\mathcal{L}/\mathcal{L})
\]
Then as explained in \cite{HT} one obtains the projections
\[
\pi':PV(\mathcal{L}) \twoheadrightarrow \mathfrak{c}_p^{-1} \rr_p
\times \rr/\ff\rr,\,\,\,\,and,\,\,\,\,\pi:PV(\mathcal{L})
\twoheadrightarrow F_p/\theta_p^{-1} \times \ff^*/\theta^{-1}
\]
We set $\mathcal{L}(\ff p):= \ff^{-1}p^{-\infty}\mathcal{L} \cap
PV(\mathcal{L}$ and for a $w \in \mathcal{L}(\ff p)$ we define
$P\phi(w):=P\phi(\pi(w),\pi'(w))$. For an integer $k \geq 1$ we
define the $\mathfrak{c}$-polarized HMF $E_k(\phi,\mathfrak{c})$ by
\[
E_k(\phi,\mathfrak{c})(\mathcal{L},\lambda,\imath):=
\frac{(-1)^{kg}\Gamma(k+s)^g}{\sqrt{(D_F)}}\sum_{w \in
\mathcal{L}(\ff
p)/\rr^{\times}}\frac{P\phi(w)}{N(w)^k|N(w)^{2s}|}\mid_{s=0}
\]
Then from \cite{Katz2,HT} we have the following proposition,

\begin{proposition}\label{q-expansion} There exists a $\mathfrak{c}$-HMF $E_{k}(\phi,\mathfrak{c})$ of level
$p^{\infty}\mathfrak{f}^{2}$ and weight $k$ such that if $k \geq 2$
or $\phi(a,0)= 0$ for all $a$ then its $q$-expansion is given by
\[
E_{k}(\phi,\mathfrak{c})(Tate_{\mathfrak{a},\mathfrak{b}}(q),\lambda_{can},\omega_{can},i_{can})=
N(\mathfrak{a})\{2^{-g}L(1-k,\phi,\mathfrak{a})\]
\[ + \sum_{ 0
\ll \xi \in \mathfrak{ab}}\sum_{(a,b) \in (\mathfrak{a} \times
\mathfrak{b})/\mathfrak{r}^{\times}, ab=\xi} \phi(a,b)
sgn(N(a))N(a)^{k-1}q^{\xi}\}
\]
where $L(s;\phi,\mathfrak{a})= \sum_{x \in
(\mathfrak{a}-0)/{\mathfrak{r}}^{\times}}\phi(x,0)sgn(N(x))^{k}|N(x)|^{-s}$.
\end{proposition}

{\textbf{Remark:}} The following remarks are in order
\begin{enumerate}
\item  In the case that the locally constant function $\phi$ is supported
on $T^{\times}:=\{\mathfrak{r}^{\times}_p \times
(\mathfrak{r}/\mathfrak{f})^{\times}\} \times
\{\mathfrak{r}^{\times}_p \times
(\mathfrak{r}/\mathfrak{f})^{\times}\}$ then the Eisenstein series
has constant term equal to zero at the cusp
$(\mathfrak{a},\mathfrak{b})$.
\item Note that the $p$-integrality of the $q$-expansion follows from
the values of the function $\phi$ and from the fact that
($\mathfrak{a},p)=1$.
\end{enumerate}

{\textbf{The Eisenstein Measure of Katz-Hida-Tilouine:}} Hida and
Tilouine extended the work of Katz to obtain measures of the Galois
group $Gal(K(\mathfrak{C}p^{\infty})/K)$ for $K$ a CM field and
$\mathfrak{C}$ an integral ideal of $K$. We describe briefly the
construction and the interpolation properties of these measures. We
start with the decomposition
$\mathfrak{C}=\mathfrak{F}\mathfrak{F_c}\mathfrak{J}$ such that
\[
\mathfrak{F}+\mathfrak{F}_c =
\mathfrak{R},\,\,\,\mathfrak{F}+\mathfrak{F}^c=\mathfrak{R},\,\,\,\mathfrak{F}_c+\mathfrak{F}_c^c=\mathfrak{R},\,\,\,\mathfrak{F}_c
\supset \mathfrak{F}^c
\]
and $\mathfrak{J}$ consists of ideals that inert or ramify in $K/F$.
We set $\ff':=\mathfrak{F}\mathfrak{J}\cap F$ and $\ff'':=
\mathfrak{F}_c\mathfrak{J}\cap F$, $\ff:=\ff' \cap \ff''=\ff'$,
$\mathfrak{s}= \mathfrak{F}_c \cap F$ and $\mathfrak{j}:=
\mathfrak{J}\cap F$. As in Hida and Tilouine we consider the
homomorphism obtained from class field theory
\[
i:\{(\rr_p^{\times} \times (\rr/\ff)^{\times} \times \rr_p^{\times}
\times (\rr/\mathfrak{s})^{\times})/\overline{\rr^{\times}} \}
\rightarrow Cl_K(\mathfrak{C}p^{\infty})
\]
We write $Cl_{K}^{-}(\mathfrak{J})$ for the quotient of
$Cl_K(\mathfrak{J})$ by the natural image of
$(\rr/\mathfrak{j})^{\times}$. If $\{\mathfrak{U}_j\}_j$ are
representatives of $Cl^{-}_K(\mathfrak{J})$, which we pick relative
prime to $p\mathfrak{C}\mathfrak{C}^c$, then we have that
$Cl_K(\mathfrak{C}p^{\infty})= \coprod_jIm(i)[\mathfrak{U}_j]^{-1}$
where $[\mathfrak{U}_j]$ the image of $\mathfrak{U}_j$ in
$Cl_K(\mathfrak{C}p^{\infty})$. We use the surjection
$(\rr/\ff)^{\times} \rightarrow (\rr/\mathfrak{s})^{\times}$ to
obtain a projection
\[
T:= \{(\rr_p^{\times} \times (\rr/\ff)^{\times} \times
\rr_p^{\times} \times (\rr/\ff)^{\times})/\overline{\rr^{\times}} \}
\twoheadrightarrow \{(\rr_p^{\times} \times (\rr/\ff)^{\times}
\times \rr_p^{\times} \times
(\rr/\mathfrak{s})^{\times})/\overline{\rr^{\times}} \}
\]
Given a continuous function $\phi$ of $Cl_K(\mathfrak{C}p^{\infty})
\cong Gal(K(\mathfrak{C}p^{\infty})/K)=:G $ we define a function
$\tilde{\phi}_j$ on $Im(i)[\mathfrak{U}_j]$ by $\tilde{\phi}_j(x):=
\phi(x[\mathfrak{U}_j^{-1}])$ and through the above projection we
view $\tilde{\phi}_j$ as function on $T$. Moreover we write
$\mathbf{N}$ for the function
\[
\mathbf{N}: (\rr_p^{\times} \times (\rr/\ff)^{\times} \times
\rr_p^{\times} \times (\rr/\ff)^{\times}) \rightarrow
\mathbb{Z}^{\times}_p
\]
given by $\mathbf{N}_k(x,a,y,b)= \prod_{\sigma \in \Sigma_p}
x_{\sigma}$. The we define functions $\phi_j$ on $(\rr_p^{\times}
\times (\rr/\ff)^{\times} \times \rr_p^{\times} \times
(\rr/\ff)^{\times})$ by
$\phi_j(x,a,y,b):=\mathbf{N}(x)^{-1}\tilde{\phi}_j(x^{-1},a^{-1},y,b)$.

In order to define the measure of Katz, Hida and Tilouine we need to
pick polarization of HBAV with complex multiplication by
$\mathfrak{R}$ and CM type $\Sigma$. We pick an element $\delta \in
K$ such that
\begin{enumerate}
\item $\delta^{c}=-\delta$ and $Im(\delta^{\sigma}) > 0$ for all
$\sigma \in \Sigma$,
\item The polarization $<u,v>:=\frac{u^cv-uv^c}{2\delta}$ on
$\mathfrak{R}$ induces the isomorphism $\mathfrak{R}
\wedge_{\mathfrak{r}} \mathfrak{R} \cong \theta^{-1}
\mathfrak{c}^{-1}$ for $\mathfrak{c}$ relative prime to $p$.
\end{enumerate}
After the above choice of $\delta$ we can attach (see \cite{HT} page
211 for details) to the fractional ideals $\mathfrak{U}_j$ of $K$ a
datum
$(X(\mathfrak{U}_j),\lambda(\mathfrak{U}_j),\imath(\mathfrak{U}_j))$
consisting of a HBAV $X(\mathfrak{U}_j)$ with CM of type
$(K,\Sigma)$, of polarization
$\mathfrak{c}\mathfrak{U}_j\mathfrak{U}^c_j$ and level structure
$\imath(\mathfrak{U}_j)$ of type $p^\infty\mathfrak{f}^2$ .

We define the measure, see \cite[pages 260-261]{Katz2} as
\[
\int_{G}\phi(g) \mu^{KHT}(g):= \sum_j \int_{T}\tilde{\phi}_j dE_j :=
\sum_j
E_1(\phi_j,\mathfrak{c}_j)(X(\mathfrak{U}_j),\lambda(\mathfrak{U}_j),\imath(\mathfrak{U}_j))
\]
where $\mathfrak{c}_j:=
\mathfrak{c}(\mathfrak{U}_j\mathfrak{U}^c_j)^{-1}$. We note here
that when $\phi$ is a character of infinite type $-k\Sigma$ then we
have that
\[
E_1(\phi_j,\mathfrak{c}_j)(X(\mathfrak{U}_j),\lambda(\mathfrak{U}_j),\imath(\mathfrak{U}_j))=
E_k(\phi_{finite,j},\mathfrak{c}_j)(X(\mathfrak{U}_j),\lambda(\mathfrak{U}_j),\imath(\mathfrak{U}_j))
\]
where $\phi_{finite,j}$ is as in \cite{Katz2} page 277 and the above
equation is explained in (5.5.7) of (loc. cit.). Here we note an
important difference of our construction from the construction of
Hida and Tilouine. We do not use the function $\phi^0_j$ in
Hida-Tilouine's notation (page 209). This is the reason why the
following measure has slightly different interpolation properties
from theirs. The reason for doing that is related with the values of
the measures $\mu_{E/F}$ and $\mu_{E/F'}$ that we will define later.
If we want these measures to take $\Z_p$ values then we have to make
sure that we put the right epsilon factors (viewed as periods) also
away from $p$.

\begin{theorem}[(Interpolation Properties)]\label{measureKHT} For a character $\chi$
of $G:=Gal(K(\mathfrak{C}p^{\infty})/K)$ of infinite type $-k\Sigma$
we have
\[
\frac{\int_{G}\chi(g)\mu_{\delta}^{KHT}(g)}{\Omega_p^{k\Sigma}} =
(\mathfrak{R}^{\times}:\rr^{\times})Local(\Sigma,\chi,\delta)\frac{(-1)^{kg}\Gamma(k)^{g}}{\sqrt{D_F}\Omega_{\infty}^{k\Sigma}}\times
\]
\[
\prod_{\mathfrak{q}|\mathfrak{FJ}}(1-\check{\chi}(\bar{\mathfrak{q}}))\prod_{\
\q | \mathfrak{F}}(1-\chi(\bar{\q}))\prod_{\mathfrak{p}\in
\Sigma_p}(1-\chi(\bar{\mathfrak{p}}))(1-\check{\chi}(\bar{\mathfrak{p}}))L(0,\chi)
\]
\end{theorem}
\begin{proof} This is in principle the measure constructed by Katz and Hida-Tilouine in
\cite{HT,Katz2}. The main difference of the above formula with the
one in Theorem 4.1 of \cite{HT} is that we do also the partial
Fourier transform for the primes that divide
$\mathfrak{F}\mathfrak{J}$ (this is why in our definition we used
$\phi$ and not $\phi^{0}$ as Hida and Tilouine do (page 209). Note
that the computations in their work are local so what we do amounts
simply moving some of the epsilon factors away from $p$ to the other
part of the functional equation (compare with theorem 4.2 in Hida
and Tilouine).
\end{proof}

We now explain the local factor $Local(\chi,\Sigma,\delta)$ that
shows up in the interpolation formula above. So we let $\chi$ be a
Gr\"{o}ssencharacter of a CM field $K$ of infinite type (after
fixing $incl(\infty): \bar{\Q} \hookrightarrow \C$)
\[
\chi_{\infty}:K^\times \rightarrow \bar{\Q} \hookrightarrow \C
\]
given by
\[
\chi_{\infty}(a) = \prod_{\sigma \in
\Sigma}\frac{1}{\sigma(a)^k}\left(\frac{\sigma(\bar{a})}{\sigma(a)}\right)^{d(\sigma)}
\]
We write $c:\A_K^\times/K^\times \rightarrow \C^\times$ for the
corresponding adelic character and we decompose it to
$c=\prod_{\sigma \in \Sigma}c_\sigma \prod_v c_v$. The infinite type
of the character can be read from the parts at infinite $c_\sigma
:\C^\times \rightarrow \C^\times$. These are given by
\[
c_\sigma(re^{i\theta})=c_\sigma(z)=\frac{z^{k+d(\sigma)}}{\bar{z}^{d(\sigma)}}=r^k
e^{i\theta(k+2d(\sigma))}
\]
Let as pick $\mathfrak{q}$, a prime ideal of $K$ which we also take
relative prime to $2$. Then we define
\[
Local(\chi,\delta)_{\mathfrak{q}}:=\frac{\hat{F}_{\mathfrak{q},1}\left(\frac{-1}{2\delta
a}\right)}{c_{\mathfrak{q}}(a)}
\]
where $a \in K$ such that
$ord_{\mathfrak{q}}(a)=ord_{\mathfrak{q}}(cond(\chi))$. Here
\[
\hat{F}_{\mathfrak{q},1}(x):=
\frac{1}{N(\mathfrak{q})^{ord_{\mathfrak{q}}cond(\chi)}}\sum_{u \in
(\mathfrak{R}/\mathfrak{q})^\times}c_{\mathfrak{q}}(y)exp(-2\pi i
\,Tr_{\mathfrak{q}}(ux))
\]
Then in the formula we have
\[
Local(\chi,\Sigma,\delta):=\prod_{\q |
\mathfrak{FJ}}Local(\chi,\delta)_{\q} \prod_{\p \in
\Sigma_p}Local(\chi,\delta)_{\p}
\]

{\textbf{The discrepancy of the $\epsilon$-factors}:} Our next goal
is to understand the relation of the local factor
$Local(\Sigma,\chi,\delta)$ appearing in the interpolation
properties of the Hida-Katz-Tilouine measure and the standard
epsilon factors of Tate-Deligne. We start by normalizing properly
the epsilon factors. We follow Tate's article \cite{Tate} for the
definition and properties of the epsilon factors of Deligne. We
denote Delinge's factor with $\epsilon_{\mathfrak{p}}(\chi,\psi,dx)$
as is defined in Tate's article \cite{Tate} where as $\psi$ we pick
the additive character of $K_{\mathfrak{p}}$ given by $exp \circ
(-Tr_{\mathfrak{p}})$ (as above in the Gauss sum appearing in Katz's
work) and $dx$ we pick the Haar measure that gives measure 1 to the
units of $\mathfrak{R}_\mathfrak{p}$. From the formula (3.6.11) in
Tate (there is a typo there!) we have that
\[
\epsilon_{\mathfrak{p}}(\chi^{-1},\psi,dx)=
c_{\mathfrak{p}}^{-1}(\alpha)N(\theta_K(\mathfrak{p}))\sum_{u \in
(\mathfrak{R}/\mathfrak{p})^\times}c_{\mathfrak{p}}(y)exp(-2\pi i
\,Tr_{\mathfrak{p}}(\frac{u}{\alpha}))
\]
where $\alpha$ is an element with $ord_{\mathfrak{p}}(\alpha) =
n(\chi) + n(\psi)$. In particular we conclude that
\[
\epsilon_{\mathfrak{p}}(\chi^{-1},\psi,dx)=
N(\mathfrak{p})^{ord_{\mathfrak{p}}cond(\chi)}c^{-1}_{\mathfrak{p}}(\delta)N(\theta_K(\mathfrak{p}))Local(\chi,\Sigma,\delta)_{\mathfrak{p}}
\]
We conclude
\begin{lemma}\label{epsilonfactors}The relation between Katz and Deligne's epsilon factors
is given by
\[
\epsilon_{\mathfrak{p}}(\chi^{-1},\psi,dx)=
N(\mathfrak{p})^{ord_{\mathfrak{p}}cond(\chi)}c^{-1}_{\mathfrak{p}}(\delta)N(\theta_K(\mathfrak{p}))Local(\chi,\Sigma,\delta)_{\mathfrak{p}}
\]
\end{lemma}
No we consider the take in the lemma above $\chi$ equal to $\chi
\psi_K^{-1}$ for $\chi$ a finite character of $K$. Then we have that
\[
\epsilon_{\mathfrak{p}}(\chi^{-1}\psi_K,\psi,dx)=\epsilon_{\mathfrak{p}}(\chi^{-1},\psi,dx)
\psi_K(\pi_{\mathfrak{p}}^{n(\chi) + n(\psi)})
\]
In particular that implies
\[
Local(\chi\psi_K^{-1},\Sigma,\delta)_{\mathfrak{p}}=N(\mathfrak{p})^{-n(\chi)}c_{\mathfrak{p}}(\delta)N(\theta_K(\mathfrak{p})^{-1})\epsilon_{\mathfrak{p}}(\chi^{-1}\psi_K,\psi,dx)
=\]
\[
=N(\mathfrak{p})^{-n(\chi)}c_{\mathfrak{p}}(\delta)N(\theta_K(\mathfrak{p})^{-1})\epsilon_{\mathfrak{p}}(\chi^{-1},\psi,dx)
\psi_K(\pi_{\mathfrak{p}}^{n(\chi) + n(\psi)}) =
\]
\[
=c_{\mathfrak{p}}(\delta)\epsilon_{\mathfrak{p}}(\chi^{-1},\psi,dx)\frac{\psi_K(\pi_{\mathfrak{p}}^{n(\chi)})}{N(\mathfrak{p})^{n(\chi)}}\frac{\psi_K(\pi_{\mathfrak{p}}^{n(\psi)})}{N(\mathfrak{p})^{n(\psi)}}
\]
where $c_{\mathfrak{p}}(\delta)$ is the value of the adelic
counterpart of $\chi\psi_K^{-1}$ at $\delta$. But as $\psi_K$ is
unramified at $\mathfrak{p}$ we have that
$c_{\mathfrak{p}}(\delta)=\psi_K(\pi_{\mathfrak{p}}^{-n(\psi)})\chi_{\mathfrak{p}}(\delta)$.
So we conclude that
\[
Local(\chi\psi_K^{-1},\Sigma,\delta)_{\mathfrak{p}}=\chi_{\mathfrak{p}}(\delta)\epsilon_{\mathfrak{p}}(\chi^{-1},\psi,dx)\left(\frac{\psi_K(\pi_{\mathfrak{p}})}{N(\mathfrak{p})}\right)^{n(\chi)}\frac{1}{N(\mathfrak{p})^{n(\psi)}}
\]

{\textbf{Remarks on the values of the measure of Katz-Hida-Tilouine
and the periods:}} In order to determine where the measures
$\mu_{E/F}$ and $\mu_{E/F'}$ defined in section 2 above take their
values we need first to explain where the measures
$\mu^{KHT}_{\psi_K,\delta}$ and $\mu^{KHT}_{\psi_{K'},\delta'}$ of
Hida-Katz-Tilouine take their values. The key point is to understand
how the interpolation formulas of these measures are related to the
period conjectures of Deligne that were proved by Blasius
\cite{Blasius}in our setting. As mentioned above in Theorem
\ref{measureKHT}, the interpolation properties of the
Katz-Hida-Tilouine measure for a character $\chi$ of
$G:=Gal(K(\mathfrak{m}p^{\infty})/K)$ of infinite type $k\Sigma$ are
\[
\frac{\int_{G}\chi(g)\mu_{\delta}^{KHT}(g)}{\Omega_p^{k\Sigma}} =
(\mathfrak{R}^{\times}:\rr^{\times})Local(\Sigma,\chi,\delta)\frac{(-1)^{kg}\Gamma(k)^{g}}{\sqrt{D_F}\Omega_{\infty}^{k\Sigma}}\times
\]
\[
\prod_{\mathfrak{q}|\mathfrak{FJ}}(1-\check{\chi}(\bar{\mathfrak{q}}))\prod_{\
\q | \mathfrak{F}}(1-\chi(\bar{\q}))\prod_{\mathfrak{p}\in
\Sigma_p}(1-\chi(\bar{\mathfrak{p}}))(1-\check{\chi}(\bar{\mathfrak{p}}))L(0,\chi)
\]
and we have fixed a Gr\"{o}ssencharacter $\psi_K$ associated to
$E/F$, unramified above $p$ and considered the measure of $G$
defined for every locally constant function $\chi$ of $G$ by
\[
\int_{G}\chi(g)\mu_{\psi_{K}\delta}^{KHT}(g):=\int_{G}\chi(g)\hat{\psi}^{-1}_K(g)\mu_{\delta}^{KHT}(g)
\]
where $\hat{\psi}_K$ is the $p$-adic avatar of $\psi_K$ constructed
by Weil. Then we consider the question in which field the algebraic
elements
$\frac{\int_{G}\chi(g)\hat{\psi}_K(g)\mu_{\delta}^{KHT}(g)}{\Omega_p^{k\Sigma}}$
belong which is equivalent to addressing the question where the
values
\[
Local(\Sigma,\chi\psi^{-1}_K,\delta)\frac{L(0,\chi\psi^{-1}_K)}{\sqrt{|D_F|}\Omega_{\infty}^{\Sigma}}
\]
exactly belong. As we will see later we can replace
$Local(\Sigma,\chi\psi^{-1}_K,\delta)$ with
$Local(\Sigma,\chi,\delta)$ as the two differ by an element in
$K^\times$. Now we note that the element $\Omega_\infty$ defined by
Katz depends only on the infinite type of $\psi_K$. However we will
assume that $\Omega_\infty$ is so selected such that
$\sqrt{|D_F|}\Omega_{\infty}^{\Sigma}$ is equal to Deligne's period
$c^{+}(\psi^{-1}_K)$. We note that this is not always possible in
Katz's construction as one is restricted to pick abelian varieties
with CM by $K$ that arise from fractional ideals of $K$. However in
our setting, as everything will be ``coming'' from an elliptic curve
$E/\Q$, we are allowed this assumption and actually we will prove
later that we are allowed to take
$\Omega^{\Sigma}_p=\Omega(E)^{g}_{p}$ and
$\Omega_{\infty}^{\Sigma}=\Omega(E)^g$ where we recall $g=[F:\Q]$.
So we may assume that
$\frac{L(0,\psi^{-1}_K)}{\sqrt{|D_F|}\Omega^{g}(E)} \in K$. As we
have mentioned above, Blasius has proved in \cite{Blasius} Deligne's
conjecture for Hecke characters of CM fields, in particular we know
that
\[
\frac{L(0,\chi\psi^{-1}_K)}{c^+(\chi\psi^{-1}_K)} \in K(\chi)
\]
where $c^+(\chi\psi_K)$ is Delinge's period for the Hecke character
$\chi\psi^{-1}_K$. In general one has that $c^+(\chi\psi^{-1}_K)
\neq c^+(\chi)c^+(\psi^{-1}_K)$. Indeed it is shown in
\cite{Schappacher} (page 107 formula 3.3.1) that
\[
\frac{c^+(\chi\psi^{-1}_K)}{c^+(\psi^{-1}_K)}=c(\Sigma,\chi)\,\,\,
\mod{ K(\chi)^\times}
\]
Here $c(\Sigma,\chi) \in (K(\chi) \otimes \bar{\Q})^\times$ is a
period associated to the finite character $\chi$ and depending on
the CM-type of the Gr\"{o}ssencharacter $\psi_K$. Actually it can be
determined, up to elements in $K(\chi)^\times$, from the following
reciprocity law. If we write $F:=K^+$ for the maximal totally real
subfield of $K$ then one can associate to the CM type $\Sigma$ the
so-called half-transfer map of Tate (see \cite{Schappacher} page
106)
\[
Ver_{\Sigma}:Gal(\bar{\Q}/F)\rightarrow Gal(\bar{\Q}/K)
\]
Then one has that
\[
(1 \otimes \tau)c(\Sigma,\chi)=(\chi\circ
Ver_{\Sigma})(\tau)c(\Sigma,\chi),\,\,\,\,\tau \in Gal(\bar{\Q}/F)
\]
So for our considerations we need to consider the question if
$Local(\chi,\Sigma,\delta)$ is equal to $c(\Sigma,\chi)$ up to
elements in $K(\chi)^{\times}$. This is in general \textbf{not} the
case. Indeed as it is explained by Blasius in \cite{Blasius2} (page
66) if we denote by $E$ the reflex field of $(K,\Sigma)$, this is a
CM field itself, then the extension
$E_{\Sigma}:=E(c(\Sigma,\chi),\,\,\chi)$, where we adjoin to $E$ the
values $c(\Sigma,\chi)$ for finite order characters $\chi$ over $K$,
is the field extension of $E$ generated by values of arithmetic
Hilbert modular functions on CM points of $\mathbb{H}^{[F:\Q]}$ of
type $(K,\Sigma)$, i.e. correspond to Hilbert-Blumenthal abelian
varietes of dimension $[F:\Q]$ with CM of type $(K,\Sigma)$. This
extension of $E$ is not included in $E\mathbb{Q}^{ab}$. However we
will see later that the elements $Local(\chi\psi_K,\Sigma,\delta)$
are almost equal to Gauss sums. In particular that implies that they
can generate over $E$ only extentions that are included in
$E\mathbb{Q}^{ab}$ (see also the comment in \cite{Schappacher} page
109). Hence in general the two ``periods'' of $\chi$ are not equal
up to elements in $K(\chi)^\times$. That implies, that in general
the measures $\frac{1}{\Omega_p(E)^{g}}\mu^{KHT}_{\psi_K,\delta}$
and $\frac{1}{\Omega_p(E)^{g'}}\mu^{KHT}_{\psi_{K'},\delta'}$ are
not elements of the Iwasawa algebras $\Z_p[[G_K]]$ and
$\Z_p[[G_{K'}]]$ respectively. However if $\chi$ is cyclotomic i.e.
$\chi(\tau g\tau^{-1})=\chi(g)$ for all $g \in G_K$ then we have the
following

\begin{lemma}\label{rationalityL} For $\chi$ cyclotomic we have
\[
\frac{\int_{G_K}\chi(g)\mu_{\psi_{K},\,\delta}^{KHT}(g)}{\Omega_p(E)^g}
\in \mathbb{Z}_p[\chi]
\]
\end{lemma}
\begin{proof}
From the interpolation properties of the measure
$\mu^{KHT}_{\psi_K,\delta}$ we have
\[
\frac{\int_{G_K}\chi(g)\mu_{\psi_{K},\,\delta}^{KHT}(g)}{\Omega_p(E)^g}=
(\mathfrak{R}^{\times}:\rr^{\times})Local(\Sigma,\chi\psi^{-1}_K,\delta)\frac{(-1)^{kg}\Gamma(k)^{g}}{\sqrt{D_F}\Omega_{\infty}(E)^{p\Sigma}}\times
\]
\[
\prod_{\mathfrak{q}|\mathfrak{FJ}}(1-\check{\chi}\check{\psi}^{-1}_K(\bar{\mathfrak{q}}))\prod_{\
\q |
\mathfrak{F}}(1-\chi\psi^{-1}_K(\bar{\q}))\prod_{\mathfrak{p}\in
\Sigma_p}(1-\chi\psi^{-1}_K(\bar{\mathfrak{p}}))(1-\check{\chi}\check{\psi}^{-1}_K(\bar{\mathfrak{p}}))L(0,\chi\psi^{-1}_K)
\]
As the measure is integral valued we have only to show that
\[
\frac{L(0,\chi\psi^{-1}_K)}{\sqrt{D_F}\Omega_{\infty}(E)^{p\Sigma}}Local(\Sigma,\chi\psi^{-1}_K,\delta)
\in \Q_p(\chi)
\]
From the discussion above we have that
$Local(\Sigma,\chi\psi^{-1}_K,\delta)$ is equal to $\prod_{\p \in
\Sigma_p}e_{\p}(\chi^{-1}\psi_K)\prod_{\q |
\mathfrak{FJ}}e_{\q}(\chi^{-1}\psi_K)$ up to elements in $K(\chi)$.
But then if we write $\mathfrak{f}_{\psi_K}$ for the conductor of
$\psi_K$ we have that $\prod_{\q |
\mathfrak{f}_{\psi_K}}e_{\q}(\psi_K)=\pm1 $ as this is the sign of
the functional equation of $E/F$. In particular up to elements in
$K(\chi)$ (as $\psi_K$ is unramified above $p$ and
$(cond(\chi),cond(\psi_K))=1$) we have that $\prod_{\p \in
\Sigma_p}e_{\p}(\chi^{-1}\psi_K)\prod_{\q |
\mathfrak{FJ}}e_{\q}(\chi^{-1}\psi_K)=\prod_{\p \in
\Sigma_p}e_{\p}(\chi^{-1})\prod_{\q |
\mathfrak{FJ}}e_{\q}(\chi^{-1})$. We write now $f_{\psi_K}$ for the
Hilbert modular form over $F$ that is induced by automorphic
induction from $\psi_K$ (i.e. the one that corresponds to the
modular elliptic curve $E/F$) and $\tilde{\chi}$ for the finite
character over $F$ whom $\chi$ is the base change of from $F$ to
$K$. Then we that up to elements in $K(\chi)$, $\prod_{\p \in
\Sigma_p}e_{\p}(\chi^{-1})\prod_{\q |
\mathfrak{FJ}}e_{\q}(\chi^{-1})=e(\tilde{\chi}^{-1})$ where
$e(\tilde{\chi}^{-1})$ the global epsilon factor of
$\tilde{\chi}^{-1}$. Moreover we have that
$L(\chi\psi^{-1}_K,0)=L(f_{\psi_K},\tilde{\chi}^{-1},1)$ (here is
crucial that $\chi$ is cyclotomic). But it is known as for example
is proved in \cite{Hida3} (page 435 Theorem I) that
\[
\frac{L(f_{\psi_K},\tilde{\chi}^{-1},1)}{\sqrt{D_F}\Omega_{\infty}(E)^{p\Sigma}}e(\tilde{\chi}^{-1})
\in \Q_p(\chi)
\]
which allows us to conclude the proof of the lemma.
\end{proof}

Actually using the full force of the results in \cite{Hida3} we have
that
\[
\left(\frac{L(f_{\psi_K},\tilde{\chi}^{-1},1)}{\sqrt{D_F}\Omega_{\infty}(E)^{p\Sigma}}e(\tilde{\chi}^{-1})\right)^\sigma
=\frac{L(f_{\psi_K},\tilde{\chi}^{-\sigma},1)}{\sqrt{D_F}\Omega_{\infty}(E)^{p\Sigma}}e(\tilde{\chi}^{-\sigma})
\]
for all $\sigma \in Gal(\bar{\Q}/\Q)$ which can be easily seen to
imply that
\[
\left(\frac{\int_{G_K}\chi(g)\mu_{\psi_{K},\,\delta}^{KHT}(g)}{\Omega_p(E)^g}\right)^{\sigma}=
\frac{\int_{G_K}(\chi(g))^\sigma\mu_{\psi_{K},\,\delta}^{KHT}(g)}{\Omega_p(E)^g}
\]
for all $\sigma \in Gal(\bar{\Q}_p/\Q_p)$.

\section{The twisted Katz-Hida-Tilouine measure}

In this section we modify the KHT-measure in the case where the
relative different is principal. The interpolation properties of the
twisted measure are going to be different with respect with the
``epsilon'' factors and with the modification of the Euler factors
at $p$. We explain now this modification. We follow the construction
that we presented above. We still consider the relative situation
$F'/F$ and the corresponding $K'/K$ extension. Under our assumption
we have that $(\xi)=\theta_{F'/F}$ where $\xi$ is a totally positive
element in $F'$. Moreover our assumptions on $F'/F$ imply that
$\theta_{F'/F}$ splits in $K'$ to $\mathfrak{P}\bar{\mathfrak{P}}$.

Over $K'$ we define the $KHT$-measure by picking instead of
$\delta'$ the element $\delta \in K \hookrightarrow K'$. Note that
since the CM type $(K',\Sigma')$ is a lift of $(K,\Sigma)$ this is a
valid choice. The polarization that the element $\delta$ induces to
the lattice $\mathfrak{R}'$ is
\[
\bigwedge^2_{\rr'}(\mathfrak{R}') \cong
\theta_{F}^{-1}\mathfrak{c}^{-1}\mathfrak{r}'
\]
if the same element, seeing as an element in $K$ induces the
polarization
\[
\bigwedge^2_{\rr}(\mathfrak{R}) \cong
\theta_{F}^{-1}\mathfrak{c}^{-1}
\]
Indeed, under our assumptions about the ramification of $F'$ and $F$
and $K_0$ we have that $\mathfrak{R}'=\mathfrak{r}'\mathfrak{R}_0$
and similarly $\mathfrak{R}=\mathfrak{r}\mathfrak{R}_0$, from which
we obtain $\mathfrak{R}'=\mathfrak{R}\mathfrak{r}'$ and the above
claim follows. With respect to this polarization we have for
fractional ideals of $K'$ of the form $\mathfrak{U} \otimes
\xi^{-1}=\mathfrak{U}\otimes \theta_{F'/F}^{-1}$ the polarization
\[
\bigwedge^2_{\rr'}(\mathfrak{U} \otimes \xi^{-1}) \cong
\theta_{F}^{-1}\mathfrak{c}^{-1}\mathfrak{U}\mathfrak{U}^c\theta_{F'/F}^{-2}=\theta_{F'}^{-1}\mathfrak{c}^{-1}\mathfrak{U}\mathfrak{U}^c\theta_{F'/F}^{-1}
\]

{\textbf{The twisted triples:}} Our twisted measure is going to be
defined again by evaluating Eisenstein series on the very CM abelian
varieties as the measure of Katz-Hida-Tilouine but we will twist
them by $\xi^{-1}$ and use the above mentioned polarization. In
particular the triples that we consider are
\begin{enumerate}
\item The abelian varieties are
$X(\mathfrak{U}_j^{\xi}):=X(\mathfrak{U}_j \otimes
\theta_{F'/F}^{-1})$.
\item The polarization $\lambda_{\delta}^{\xi}(\mathfrak{U}_j\otimes
(\theta_{F'/F}^{-1})) := \lambda_{\delta}(\mathfrak{U}_j\otimes
(\theta_{F'/F}^{-1}))$ i.e. the one defined above and
\item The $p^\infty\ff^2$-arithmetic structure is obtained from an
isomorphism $X(\mathfrak{U}_j) \cong X(\mathfrak{U}_j \otimes
\xi^{-1})$. We will amplify on this below.
\end{enumerate}

We then define the twisted measure as follows
\[
\int_{G'}\phi(g) \mu_{\delta,\xi}^{KHT,tw}(g):= \sum_j
\int_{T}\tilde{\phi}_j dE_j := \sum_j
E_1(\phi_j,\mathfrak{c}_j)(X(\mathfrak{U}^{\xi}_j),\lambda^{\xi}_{\delta}(\mathfrak{U}_j\otimes
\theta_{F'/F}^{-1}),\imath^{\xi}(\mathfrak{U}_j \otimes
\theta_{F'/F}^{-1}))
\]
with
$\mathfrak{c}_j:=\mathfrak{c}(\mathfrak{U}\mathfrak{U}^c)^{-1}\theta_{F'/F}$.
We next explore the interpolation properties of the twisted measure.
Let us write $cond(\chi)_p=\prod_{\p_j \in
\Sigma'_p}\p^{a_j}_j\bar{p}^{b_j}_{j}$ for the $p$-part of the
conductor of $\chi$. We define $e_j:=ord_{\p_j}\xi$ for all $\p_j
\in \Sigma'_p$. We have already described a decomposition
$\mathfrak{C}=\mathfrak{F}\mathfrak{F}_c\mathfrak{J}$. For $\q_j |
\mathfrak{FJ}$ we define $d_j:=ord_{\q_j}\xi$ and we write
$cond(\chi)_{\mathfrak{FJ}}=\prod_{\q_j}\q^{\ell_j}_j$.

\begin{proposition}[Interpolation Properties of the ``twisted'' Katz-Hida-Tilouine measure]\label{measureKHTtwP} For a character $\chi$
of $G':=Gal(K'(\mathfrak{C}p^{\infty})/K')$ of infinite type
$-k\Sigma'$ we have
\[
\frac{\int_{G'}\chi(g)\mu_{\delta,\xi}^{KHT,tw}(g)}{\Omega_p^{k\Sigma'}}
=
({\mathfrak{R}'}^{\times}:{\rr'}^{\times})Local(\Sigma',\chi,\delta,\xi)\prod_{a_j=0}\chi(\p_j)^{-e_j}
\prod_{\ell_j=0}\chi(\q_j)^{-d_j}\times
\]
\[
\prod_{\mathfrak{q}_j|\mathfrak{J}}(1-\check{\chi}(\mathfrak{q}_j))\left(\prod_{\mathfrak{q}_j|\mathfrak{F}}(1-\check{\chi}(\bar{\mathfrak{q}}_j))(1-\chi(\bar{\mathfrak{q}}_j))\right)\left(\prod_{\p_j
\in
\Sigma'_p}(1-\check{\chi}(\bar{\p}_j))(1-\chi(\bar{\p}_j))\right)
\]
\[
\frac{(-1)^{kg'}\Gamma(k)^{g'}}{\sqrt{D_{F'}}\Omega_{\infty}^{k\Sigma'}}
\times L(0,\chi)
\]
\end{proposition}

Here the factor $Local(\Sigma',\chi,\delta,\xi)$ is a modification
of the local factor of the measure of Katz-Hida-Tilouine and it will
be defined in the proof of the proposition. But before we proceed to
the proof of the above proposition we must explain a little bit more
the $p^\infty$-part of the given arithmetic structure of twisted
HBAV used in the above proposition. As in Katz we use the ordinary
type $\Sigma_p$  to obtain an isomorphism

\[
\mathfrak{R}'\otimes_{\Z} \Z_p  \cong \prod_{\mathfrak{p} \in
\Sigma_p} \mathfrak{R}'_{\mathfrak{p}} \times \prod_{\mathfrak{p}
\in \bar{\Sigma}_p} \mathfrak{R}'_{\mathfrak{p}}\cong \rr'_p \times
\rr'_p
\]
And similarly for any fractional ideal $\mathfrak{U}$ of
$\mathfrak{R}'$ relative prime to $p$ we can identify $\mathfrak{U}
\otimes \Z_p = \mathfrak{R}' \otimes \Z_p$ in $K' \otimes \Z_p$. In
particular we have an isomorphism for such ideals
\[
\mathfrak{U} \otimes_{\Z} \Z_p  \cong \prod_{\mathfrak{p} \in
\Sigma_p} \mathfrak{R}'_{\mathfrak{p}} \times \prod_{\mathfrak{p}
\in \bar{\Sigma}_p} \mathfrak{R}'_{\mathfrak{p}}\cong \rr'_p \times
\rr'_p
\]
Then as Katz explains (see \cite{Katz2} page 265 and lemma 5.7.52)
the $p^\infty$ structure of $X(\mathfrak{U})$ is defined by picking
the isomorphism
\[
\rr'_p \cong \theta^{-1}_{F'} \otimes \Z_p
\]
given by $x \mapsto \delta_0 x$, where $\delta_0$ is the image of
$(2\delta')^{-1}$ in $K'_p$ and using it to define the injection
\[
\theta^{-1}_{F'} \otimes \Z_p \hookrightarrow \mathfrak{U}
\otimes_{\Z} \Z_p \cong \rr'_p \times \rr'_p
\]
using the isomorphism in the first component. Now the $p^\infty$
structure of the twisted varieties $\mathfrak{U} \otimes \xi^{-1}$
is defined using the isomorphisms
\[
(\mathfrak{U}\otimes \xi^{-1}) \otimes_{\Z} \Z_p  \cong
\prod_{\mathfrak{p} \in \Sigma_p} \frac{1}{\xi}
\mathfrak{R}'_{\mathfrak{p}} \times \prod_{\mathfrak{p} \in
\bar{\Sigma}_p} \frac{1}{\xi} \mathfrak{R}'_{\mathfrak{p}}\cong
\frac{1}{\xi} \rr'_p \times \frac{1}{\xi} \rr'_p
\]
and picking the isomorphism
\[
\frac{1}{\xi} \rr'_p=\theta^{-1}_{F'/F} \otimes \rr'_p \cong
\theta^{-1}_{F'} \otimes \Z_p
\]
given by $x \mapsto x \delta^{-1}_0$ where $\delta_0$ is the image
of $\delta$ in $\prod_{\p \in \Sigma_p'}K'_\p \cong
\prod_{\p}F'_\p$. Now we proceed to the proof of the proposition on
the interpolation properties of the twisted Katz-Hida-Tilouine
measure.

\begin{proof} We will follow closely the proof of Katz in
\cite{Katz2}. Actually we will mainly indicate the differences of
our setting from his setting. We start with the following
observation. As the computations are local in nature (see also the
remark of Hida and Tilouine in \cite{HT} page 214) it is enough to
prove the theorem for characters $\chi$ of $G'$ that ramify only at
$p$.

Now we split the proof in two cases. We first consider the case
where the character $\chi$ is ramified in all primes $\p \in
\Sigma'_p$ and then we generalize.

{\textbf{Special Case: $\chi$ ramified at all $\p$ in $\Sigma'_p$:}}
We follow Katz \cite{Katz2} as in page 279 and use his notation. We
write the conductor of the character $\chi$, $cond(\chi) =
\prod_{i}\mathfrak{p}_i^{a_i}\bar{\mathfrak{p}}_i^{b_i}$. Moreover
we decompose $(\xi)=\mathfrak{P}\bar{\mathfrak{P}}$ as ideals in
$K'$. We also write
$\mathfrak{P}\prod_{i}\mathfrak{p}_i^{a_i}=(\alpha)\mathfrak{B}$ for
$\alpha \in {K'}^{\times}$ and $\mathfrak{B}$ prime to $p$. In the
case that we consider we have $a_i \geq 1$ for all $i$. From the
definition of the $p^\infty$-structure we have that the function
$P_\delta\tilde{F}$ is supported in
\[
(\prod_{i}\mathfrak{p}_i^{-a_i})\mathfrak{U}_j\mathfrak{P}^{-1} =
(\alpha^{-1})\mathfrak{B}^{-1}\mathfrak{U}_j
\]
In particular the computations of Katz for the twisted values now
read,
\[
\sum_{j=1}^h \chi(\mathfrak{U}_j)^{-1} \sum_{a \in
\mathfrak{U}_j(\xi^{-1})[\frac{1}{p}]\cap
PV_p(\mathfrak{U}_j(\xi^{-1}))}\frac{P_\delta\tilde{F}(a)}{\prod_{\sigma}\sigma(a)^k|N^{K'}_{\Q}(a)|^s}=
\]
\[
=\sum_{j=1}^h \chi(\mathfrak{U}_j)^{-1}\sum_{a \in
\mathfrak{B}^{-1}\mathfrak{U}_j}\frac{P_\delta\tilde{F}(\alpha^{-1}a)}{\prod_{\sigma}\sigma(\alpha^{-1}a)^k
|N^{K'}_{\Q}(\alpha^{-1}a)|^s}=
\]
\[
=\sum_{j=1}^h \chi(\mathfrak{U}_j)^{-1}\sum_{a \in
\mathfrak{B}^{-1}\mathfrak{U}_j}\frac{P_{\delta}\tilde{F}(\alpha^{-1})\chi_{finite}(a)}{\prod_{\sigma}\sigma(\alpha^{-1}a)^k
|N^{K'}_{\Q}(\alpha^{-1}a)|^s}=
\]
\[
=\left(P_{\delta}\tilde{F}(\alpha^{-1})|N^{K'}_{\Q}(\alpha)|^s\prod_{\sigma}\sigma(\alpha)^k\right)\sum_{j=1}^h
\chi(\mathfrak{U}_j)^{-1}\sum_{a \in
\mathfrak{B}^{-1}\mathfrak{U}_j}\frac{\chi_{finite}(a)}{\prod_{\sigma}\sigma(a)^k
|N^{K'}_{\Q}(a)|^s}
\]

There is a special case where it is easy to see the difference of
the new factors with those of Katz. Let us assume that for the
decomposition $\theta_{F'/F}=\mathfrak{P}\bar{\mathfrak{P}}$ there
exists $\zeta \in K'$ so that $\mathfrak{P}=(\zeta)$. We define
$\alpha' \in {K'}^\times$ as in Katz by
$\prod_{i}\mathfrak{p}_i^{a_i}=(\alpha')\mathfrak{B}'$ for
$\mathfrak{B}'$ prime to $p$ and we compare
\[
Local(\Sigma',\chi,\delta,\xi)_p:=
\frac{P_{\delta}\tilde{F}(\alpha^{-1})}{\chi(\mathfrak{B})}\prod_{\sigma}\sigma(\alpha)^k
\]
against the local factor of Katz
\[
\frac{P_{\delta'}\tilde{F}({\alpha'}^{-1})}{\chi(\mathfrak{B}')}\prod_{\sigma}\sigma(\alpha')^k
\]
We consider
\[
\frac{\frac{P_{\delta}\tilde{F}(\alpha^{-1})}{\chi(\mathfrak{B})}\prod_{\sigma}\sigma(\alpha)^k}{\frac{P_{\delta'}\tilde{F}({\alpha'}^{-1})}{\chi(\mathfrak{B}')}\prod_{\sigma}\sigma(\alpha')^k}=
\frac{P_{\delta}\tilde{F}(\alpha^{-1})}{P_{\delta'}\tilde{F}({\alpha'}^{-1})}\times
\chi(\mathfrak{B}'\mathfrak{B}^{-1})\times
\prod_{\sigma}\sigma\left(\frac{\alpha}{\alpha'}\right)^{k}
\]

Note that from our assumptions $\xi = \zeta \bar{\zeta}$ hence we
have $\alpha=\alpha'\zeta$. This implies
\[
\frac{P_{\delta}\tilde{F}(\alpha^{-1})}{P_{\delta'}\tilde{F}({\alpha'}^{-1})}=\frac{\prod_{\mathfrak{p}
\in
\Sigma'_p}\hat{F}_{\mathfrak{p},\delta}(\alpha^{-1})F_{\bar{\mathfrak{p}}}(\alpha^{-1})}{\prod_{\mathfrak{p}
\in
\Sigma'_p}\hat{F}_{\mathfrak{p},\delta'}({\alpha'}^{-1})F_{\bar{\mathfrak{p}}}({\alpha'}^{-1})}=\prod_{\mathfrak{p}
\in
\Sigma'_p}\chi_{\mathfrak{p}}(\bar{\zeta})\chi_{\bar{\mathfrak{p}}}(\zeta^{-1})
\]
and  $\mathfrak{B}=\mathfrak{B}'$ and
$\prod_{\sigma}\sigma\left(\frac{\alpha}{\alpha'}\right)^{k}=\prod_{\sigma}\sigma(\zeta)^{k}$.

 {\textbf{The general case:}} Now we consider the case where
some of the $a_i$'s in $cond(\chi) =
\prod_{i}\mathfrak{p}_i^{a_i}\bar{\mathfrak{p}}_i^{b_i}$ are zero.
We start by stating the following (see \cite{Katz2} page 282 or
\cite{HT} page 209),
\[
\int_{\mathfrak{R}^{\times}_{\p}}\psi_{\delta'}(xy)dy=I_{\mathfrak{R}_{\p}}(x)-\frac{1}{N\p}I_{\p^{-1}\mathfrak{R}_{\p}}(x)
\]
where $\psi_{\delta'}$ is the additive character of $K_{\p}$ given
by
\[
\psi_{\delta'}(x):=exp \circ Tr_{\p}\left(\frac{x}{\delta'}\right)
\]
In particular if we denote by $\psi_\delta$ the additive character
\[
\psi_{\delta}(x):=exp \circ Tr_{\p}\left(\frac{x}{\delta}\right)
\]
we have
\[
\int_{\mathfrak{R}^{\times}_{\p}}\psi_{\delta}(xy)dy=I_{\mathfrak{R}_{\p}}(x\xi)-\frac{1}{N\p}I_{\p^{-1}\mathfrak{R}_{\p}}(x\xi)
\]
where we recall $\xi = \frac{\delta'}{\delta}$ up to elements in
${\mathfrak{R}'_{\p}}^{\times}$. Now we follow the computations of
Katz as in (\cite{Katz2} page 281-282). We use the same notation as
in Katz. In our setting after the observation above we have that the
function $P\tilde{F}$ is supported in
\[
\prod_{a_i \geq 1}\p_i^{-a_i}(\prod_{a_j \geq
1}\p_j^{-e_j})(\prod_{a_j=0}\p_j^{-1-e_j})\mathfrak{U}_i=(\alpha^{-1})\mathfrak{B}^{-1}(\prod_{a_j=0}\p_j^{-1-e_j})\mathfrak{U}_i
\]
where $\alpha$ relative prime to the $\p_i$'s with $a_i \geq 1$,
$\mathfrak{B}$ prime to $p$ and $e_j:=ord_{\p_j}\xi$. From the
observation above we have that for $a \in
\mathfrak{B}^{-1}(\prod_{a_j=0}\p_j^{-1-e_j})\mathfrak{U}_i$ we have
\[
P\tilde{F}(\alpha^{-1}a)=P_{\delta}F(\alpha^{-1})
\chi_{2,finite}(a)\prod_{a_j=0}\widehat{char}(\p_j^{1+e_j})(a)
\]
where
\[
\widehat{char}(\p_j^{1+e_j})(a)=\left\{
                                  \begin{array}{ll}
                                    1-\frac{1}{N\p_j}, & \hbox{if $ord_{\p_j}(a) \geq -e_j$;} \\
                                    -\frac{1}{N\p_j}, & \hbox{if $ord_{\p_j}(a)=-e_j-1$.}
                                  \end{array}
                                \right.
\]
Following Katz (note a typo in Katz's definition! compare 5.5.31
with 5.5.35) we extend the above function to the set $\mathbf{I}$ of
fractional ideals $I$ of $K'$ of the form
\[
I=(\prod_{a_j=0}\p_j^{-1-e_j})\mathfrak{P}
\]
where $\mathfrak{P}$ is an integral ideal, prime to those $\p_i$
with $a_i \neq 0$ and to all $\bar{\p}_k$ by
\[
\widehat{char}(\p_j^{1+e_j})(I)=\left\{
                                  \begin{array}{ll}
                                    1-\frac{1}{N\p_j}, & \hbox{if $I\p_j^{e_j}$ is integral;} \\
                                    -\frac{1}{N\p_j}, & \hbox{if not.}
                                  \end{array}
                                \right.
\]
Following Katz's computations we have that the values that we are
interested in are
\[
\sum_{j=1}^h \chi(\mathfrak{U}_j)^{-1}\sum_{a \in
\mathfrak{B}^{-1}(\prod_{a_j=0}(\p_j^{-1-e_j}))\mathfrak{U}_j}\frac{P_\delta\tilde{F}(\alpha^{-1}a)}{\prod_{\sigma}\sigma(\alpha^{-1}a)^k
|N^{K'}_{\Q}(\alpha^{-1}a)|^s}=
\]
\[
\left(\frac{P_{\delta}\tilde{F}(\alpha^{-1})}{\chi(\mathfrak{B})}\prod_{\sigma}\sigma(\alpha)^k\right)
\sum_{I_0 \in
\mathbf{I}(p)}\frac{\chi(I_0)}{N(I_0)^s}\prod_{a_j=0}\sum_{n \geq
-1-e_j}\frac{\chi_2(\p_j)^n}{N(p_j)^{ns}}\widehat{char}(\p_j^{1+e_j})(\p_j^n)
\]
As in Katz we compute the inner sum
\[
\sum_{n=-1-e_j}^{\infty}\frac{\chi_2(\p_j)^n}{N(\p_j)^{ns}}\widehat{char}(\p_j^{1+e_j})(\p_j^n)=\frac{-1}{N(\p_j)}\frac{\chi_2(\p_j)^{-1-e_j}}{N(\p_j)^{(-1-e_j)s}}+\left(1-\frac{1}{N(\p_j)}\right)\sum_{n=-e_j}^{\infty}\frac{\chi_2(\p_j)^n}{N(\p_j)^{ns}}
\]
\[
=\sum_{n=-e_j}^{\infty}\frac{\chi_2(\p_j)^n}{N(\p_j)^{ns}}-\frac{1}{N(\p_j)}\left(\frac{\chi_2(\p_j)^{-1-e_j}}{N(\p_j)^{(-1-e_j)s}}+\sum_{n=-e_j}^{\infty}\frac{\chi_2(\p_j)^n}{N(\p_j)^{ns}}\right)
\]
\[
=\sum_{n=-e_j}^{\infty}\frac{\chi_2(\p_j)^n}{N(\p_j)^{ns}}-\frac{1}{N(\p_j)}\sum_{n=-1-e_j}^{\infty}\frac{\chi_2(\p_j)^n}{N(\p_j)^{ns}}
\]
\[
=\left(1-\frac{1}{N(\p_j)}\frac{\chi_2(\p_j)^{-1}}{N(\p_j)^{-s}}\right)\sum_{n=-e_j}^{\infty}\frac{\chi_2(\p_j)^n}{N(\p_j)^{ns}}
\]
\[
=\left(1-\frac{N(\p_j)^{s}}{\chi_2(\p_j)N(\p_j)}\right)\frac{\chi_2(\p_j)^{-e_j}}{N(\p_j)^{-e_js}}\sum_{n=0}^{\infty}\frac{\chi_2(\p_j)^n}{N(\p_j)^{ns}}
\]
\[
=\left(1-\frac{N(\p_j)^{s}}{\chi_2(\p_j)N(\p_j)}\right)\frac{\chi_2(\p_j)^{-e_j}}{N(\p_j)^{-e_js}}\left(1-\chi_2(\p_j)N(\p_j)^{-s}\right)^{-1}
\]
\[
=\left(1-N(\p_j)^{s}\check{\chi}_2(\bar{\p}_j)\right)\frac{\chi_2(\p_j)^{-e_j}}{N(\p_j)^{-e_js}}\left(1-\chi_2(\p_j)N(\p_j)^{-s}\right)^{-1}
\]
So we conclude,
\[
\sum_{j=1}^h \chi(\mathfrak{U}_j)^{-1}\sum_{a \in
\mathfrak{B}^{-1}(\prod_{a_j=0}(\p_j^{-1-e_j}))\mathfrak{U}_j}\frac{P_\delta\tilde{F}(\alpha^{-1}a)}{\prod_{\sigma}\sigma(\alpha^{-1}a)^k
|N^{K'}_{\Q}(\alpha^{-1}a)|^s}=
\]
\[
=\left(\frac{P_{\delta}\tilde{F}(\alpha^{-1})}{\chi(\mathfrak{B})}\prod_{\sigma}\sigma(\alpha)^k\right)L(s,\chi_1)\prod_{a_j=0}\left(\frac{1-N(\p_j)^{s}\check{\chi}_2(\bar{\p}_j)}{(1-\chi_2(\p_j)N(\p_j)^{-s})}\times\frac{\chi_2(\p_j)^{-e_j}}{N(\p_j)^{-e_js}}\right)
\]
whose value at $s=0$ is equal to
\[
\left(\frac{P_{\delta}\tilde{F}(\alpha^{-1})}{\chi(\mathfrak{B})}\prod_{\sigma}\sigma(\alpha)^k\right)L(0,\chi_1)\prod_{a_j=0}\left(\frac{1-\check{\chi}_2(\bar{\p}_j)}{(1-\chi_2(\p_j))}\times\chi_2(\p_j)^{-e_j}\right)
\]
But
$L(s,\chi_1)=L(s,\chi)\prod_{\p_i}\left(1-\chi(\p_i)N(\p_i)^{-s}\right)\left(1-\chi(\bar{\p}_i)N(\bar{\p}_i)^{-s}\right)$
which allow us to conclude that the values are equal to
\[
\left(\frac{P_{\delta}\tilde{F}(\alpha^{-1})}{\chi(\mathfrak{B})}\prod_{\sigma}\sigma(\alpha)^k\right)L(0,\chi)\left(\prod_{\p_j
\in
\Sigma'_p}(1-\check{\chi}(\bar{\p}_j))(1-\chi(\bar{\p}_j))\right)\times
\prod_{a_j=0}\chi(\p_j)^{-e_j}
\]
\end{proof}

\section{The relative setting: Congruences between Eisenstein series}

Now we consider the following relative setting. We consider as in
the introduction a totally real field galois extension $F'$ of $F$
of degree $p$ ramified only at $p$ and write $\Gamma=Gal(F'/F)$. We
fix ideals $\mathfrak{a},\mathfrak{b},\mathfrak{c}$ and
$\mathfrak{f}$ of $F$ and consider also the corresponding ideals in
$F'$, that is their natural image under $F \hookrightarrow F'$. We
write $T'$ and $T'^{\times}$ for the corresponding spaces in the
$F'$ setting that we have introduced for the $F$ setting. We note
that $\Gamma$ operates naturally on the spaces $T'$ and
$T'^{\times}$. Moreover the embedding $F \hookrightarrow F'$ induces
a natural diagonal embedding $\mathbb{H}^{[F:\Q]} \hookrightarrow
\mathbb{H}^{[F':\Q]}$ with the property that the pull back of a
Hilbert modular form of $F'$ is a Hilbert modular form of $F$. We
need to make this last remark a little bit more explicit.

{\textbf{The Tate-Abelian Scheme and the modular interpretation of
the diagonal embedding:}} We would like now to describe the
geometric meaning of the diagonal embedding. We follow the book of
Hida \cite{Hida2} as in chapter 4 (and especially section 4.1.5) and
the notation there.

For fractional ideals $\mathfrak{a}$ and $\mathfrak{b}$ of the
totally real field $F$ and a ring $R$ we define the ring
$R[[(\mathfrak{ab})_+]]$ with $(\mathfrak{ab})_+:=\mathfrak{ab} \cap
F_+$ to be the ring of formal series
\[
R[[(\mathfrak{ab})_+]]:= \{a_0 \sum_{\xi \in (\mathfrak{ab})_+}
a_{\xi}q^{\xi}|\,\,a_\xi \in R\}
\]
We pick the multiplicative set $q^{(\mathfrak{ab})_+}:=\{q^\xi| \xi
\in (\mathfrak{ab})_+\}$ and define $R\{\mathfrak{ab}\}$ as the
localization of $R[[(\mathfrak{ab})_+]]$ to this multiplicative set.
Then as explained in Hida the Tate semi-abelian scheme
$Tate_{\mathfrak{a},\mathfrak{b}}(q)$ is defined over the ring
$R\{\mathfrak{ab}\}$ (with $R$ depending on the extra level
structure that we impose) by the algebraization of the rigid
analytic variety
\[
(\mathbf{G}_m \otimes \mathfrak{a}^{-1}\theta^{-1}_F) /
q^{\mathfrak{b}}
\]

Let $X$ be a HBAV over a ring $R$ with real multiplication by
$\mathfrak{r}$. We may define a HBAV $X'$ over $R$ with real
multiplication by $\mathfrak{r}'$ by considering the functor from
schemes $S$ over $R$ to $\mathfrak{r}'$ modules defined by
\[
S \mapsto X'(S):= X(S) \otimes_{\mathfrak{r}} \theta^{-1}_{F'/F}
\]
We let $\mathfrak{c}:=\mathfrak{a}\mathfrak{b}^{-1}$ and consider
the effect of our map on the Tate curve
$Tate_{\mathfrak{a},\mathfrak{b}}(q)$. That is we consider the HBAV
with real multiplication by $\mathfrak{r}'$ defined by
$Tate_{\mathfrak{a},\mathfrak{b}}(q) \otimes_{\mathfrak{r}}
\theta_{F'/F}^{-1} = (\mathbf{G}_m \otimes \theta^{-1}_F) /
q^{\mathfrak{b}}\otimes_{\mathfrak{r}} \theta_{F'/F}^{-1} $. We
consider the map $tr_{F'/F}:
R\{\mathfrak{a}\mathfrak{b}\theta_{F'/F}^{-1}\} \rightarrow
R\{\mathfrak{a}\mathfrak{b}\}$ given by $q^\alpha \mapsto
q^{tr_{F'/F}(\alpha)}$. Then we have,
\begin{lemma}\label{diagonalmapL}
\[
Tate_{\mathfrak{a}\mathfrak{r}',\mathfrak{b}\theta^{-1}_{F'/F}}(q)
\times_{R\{\mathfrak{a}\mathfrak{b}\theta_{F'/F}^{-1}\}}
R\{\mathfrak{a}\mathfrak{b}\} \cong
Tate_{\mathfrak{a},\mathfrak{b}}(q) \otimes_{\mathfrak{r}}
\theta_{F'/F}^{-1}
\]
\end{lemma}
\begin{proof} Even though the lemma holds in general we are going to use it
while working over number fields. Hence after fixing embeddings in
the complex numbers we may just prove it over $\C$. Over the complex
numbers this follows easily by observing that
$Tate_{\mathfrak{a},\mathfrak{b}}(q)$ corresponds to that lattice
$2\pi i (\mathfrak{b}z + \mathfrak{a}^{-1}\theta^{-1}_{F})$ for $z
\in \mathbb{H}_F$ and hence $Tate_{\mathfrak{a},\mathfrak{b}}(q)
\otimes_{\mathfrak{r}} \theta_{F'/F}^{-1}$ to the lattice
\[
2\pi i (\mathfrak{b}z +
\mathfrak{a}^{-1}\theta^{-1}_{F})\otimes_{\mathfrak{r}}
\theta_{F'/F}^{-1} =2\pi i (\mathfrak{b}\theta_{F'/F}^{-1}z' +
\mathfrak{a}^{-1}\theta^{-1}_{F'})
\]
with $z' \in \mathbb{H}_{F'}$ the image of $z$ under the diagonal
embedding $\mathbb{H}_F \hookrightarrow \mathbb{H}_{F'}$ induced
from $F \hookrightarrow F'$. Moreover in this case the map
$tr_{F'/F}: R\{\mathfrak{a}\mathfrak{b}\theta_{F'/F}^{-1}\}
\rightarrow R\{\mathfrak{a}\mathfrak{b}\}$ given by $q^\alpha
\mapsto q^{tr_{F'/F}(\alpha)}$ corresponds to setting the
indeterminate $q:=exp(Tr_{F'}(z')):=exp(\sum_{\sigma \in
\Sigma'}z'_{\sigma})$ (where $\sigma \in \Sigma'$ the embeddings
$\sigma : F' \hookrightarrow \C$ and $z'=(z'_{\sigma}) \in
\mathbb{H}^{[F':\Q]}$) equal to the indeterminate
$q=exp(Tr_{F'}(\Delta(z)))$ for $\Delta :\mathbb{H}^{[F:\Q]}
\hookrightarrow \mathbb{H}^{[F':\Q]}$, the diagonal map. In
particular that implies that the complex points of
$Tate_{\mathfrak{a}\mathfrak{r}',\mathfrak{b}\theta^{-1}_{F'/F}}(q)
\times_{R\{\mathfrak{a}\mathfrak{b}\theta_{F'/F}^{-1}\}}
R\{\mathfrak{a}\mathfrak{b}\}$ correspond to the lattice $2\pi i
(\mathfrak{b}\theta_{F'/F}^{-1}z' +
\mathfrak{a}^{-1}\theta^{-1}_{F'})$ for $z'=\Delta(z)$.

\end{proof}
We can use the above lemma to study the effect of the diagonal
embedding to the the $q$-expanion, that is to the values of Hilbert
modular forms on the Tate abelian scheme. For a
$\mathfrak{c}\theta_{F'/F}$-HMF $\phi$ of $F'$ we have that
\[
\phi(Tate_{\mathfrak{a},\mathfrak{b}}(q) \otimes_{\mathfrak{r}}
\theta_{F'/F}^{-1})=\phi(Tate_{\mathfrak{a}\mathfrak{r}',\mathfrak{b}\theta^{-1}_{F'/F}}(q)
\times_{R\{\mathfrak{a}\mathfrak{b}\theta_{F'/F}^{-1}\}}
R\{\mathfrak{a}\mathfrak{b}\})=\]
\[
=\phi(Tate_{\mathfrak{a}\mathfrak{r}',\mathfrak{b}\theta_{F'/F}^{-1}}(q))\times_{R\{\mathfrak{a}\mathfrak{b}\theta_{F'/F}^{-1}\}}
R\{\mathfrak{a}\mathfrak{b}\}
\]

The next question that we need to clarify is what is happening under
this diagonal map for an HBAV with real multiplication by
$\mathfrak{r}$ that has CM by $\mathfrak{R}$, the ring of integers
of a totally imaginary quadratic extension $K$ of $F$. It is well
known that up to isomorphism these are given by the fractional
ideals of $K$. Let us write $\mathfrak{U}$ for one of these and
$X(\mathfrak{U})$ for the corresponding HBAV with CM by
$\mathfrak{R}$. We see that the above map gives us the HBAV
$X(\mathfrak{U}) \otimes_{\mathfrak{r}} \theta^{-1}_{F'/F}$ with
real multiplication by $\mathfrak{r}'$. We set $K'=KF'$ and write
$\mathfrak{R}'$ for its ring of integers. Then we have,
\begin{lemma} Assume that $\mathfrak{R}'=\mathfrak{R}\mathfrak{r}'$. Then the HBAV $X(\mathfrak{U}) \otimes_{\mathfrak{r}} \theta^{-1}_{F'/F}$
has CM by $\mathfrak{R}'$ and it corresponds to the fractional ideal
$\mathfrak{U}\mathfrak{D}^{-1}$ with $\mathfrak{D} =
\theta_{F'/F}\mathfrak{R}'$.
\end{lemma}
\begin{proof} We write $K=F(d)$ and then $K'=F'(d)$.
In particular since $X(\mathfrak{U})$ has CM by $K$ we conclude that
$X(\mathfrak{U}) \otimes_{\mathfrak{r}} \theta^{-1}_{F'/F}$ has CM
by $K'$ as we have $d \in End(X(\mathfrak{U})) \hookrightarrow
End(X(\mathfrak{U})) \otimes_{\mathfrak{r}} \theta^{-1}_{F'/F})$.
Moreover we have

\[X(\mathfrak{U}) \otimes_{\mathfrak{r}} \theta^{-1}_{F'/F} = X(\mathfrak{U}) \otimes_{\mathfrak{r}}
\mathfrak{r}' \otimes_{\mathfrak{r}'} \theta^{-1}_{F'/F} =
X(\mathfrak{U}\mathfrak{R}') \otimes_{\mathfrak{r}'}
\theta^{-1}_{F'/F} =
X(\mathfrak{U}\mathfrak{R}')/(X(\mathfrak{U}\mathfrak{R}')[\theta_{F'/F}])\]

But we have that
$X(\mathfrak{U}\mathfrak{R}')/(X(\mathfrak{U}\mathfrak{R}')[\theta_{F'/F}])
= X(\mathfrak{U}\theta^{-1}_{F'/F}\mathfrak{R}')$ which concludes
the proof as a fractional ideal of $K'$ has CM by $\mathfrak{R}'$.

\end{proof}

We remark that the condition of the lemma,
$\mathfrak{R}'=\mathfrak{R}\mathfrak{r}'$ holds in our setting.
Indeed we know that $\mathfrak{R}=\mathfrak{R}_0\mathfrak{r}$ as
$[K:\Q]=[K_0:\Q][F:\Q]$ and $F/\Q$ and $K_0/\Q$ have disjoint
ramification. Similarly we have
$\mathfrak{R}'=\mathfrak{R}_0\mathfrak{r}'$. But then we have
$\mathfrak{R}'=\mathfrak{R}_0 \mathfrak{r}'=\mathfrak{R}_0
\mathfrak{r}\mathfrak{r}'=\mathfrak{R}\mathfrak{r}'$.

The key proposition is now is the following which later will allow
us to compare the measures of Katz-Hida-Tilouine over $K$ and $K'$.

\begin{proposition}(Congruences)\label{CongruencesP} Let $\mathfrak{c}$ be a fractional ideal of $F$ relative prime to $p$.
We have the congruences of Eisenstein series
\[
res_\Delta(E_{k}(\phi',\mathfrak{c}\theta_{F'/F})) \equiv
Frob_{p}(E_{pk}(\phi,\mathfrak{c}) \mod{p}
\]
where $\phi := \phi' \circ ver$ and $\phi'$ a locally constant
$\Z_p$-valued function on ${\rr'}_p^{\times} \times
({\rr'}/\ff)^{\times} \times {\rr'}_p^{\times} \times
({\rr'}/\mathfrak{\ff})^{\times}$ with $\phi^{\gamma}=\phi$ for all
$\gamma \in \Gamma$.
\end{proposition}

\begin{proof} We consider the cusp $(\rr',\mathfrak{b}\theta^{-1}_{F'/F})$ for
$\mathfrak{b}$ a fractional ideal of $F$ equal to
$\mathfrak{c}^{-1}$. From Proposition \ref{q-expansion} we know that
the $q$-expansion of the Eisenstein series
$E_{k}(\phi',\mathfrak{c}\theta_{F'/F})$ at the cusp
$(\rr',\mathfrak{b}\theta_{F'/F}^{-1})$ is given by
\[
E_{k}(\phi',\mathfrak{c}\theta_{F'/F})(Tate_{\rr',\mathfrak{b}\theta_{F'/F}^{-1}}(q),\lambda_{can},\omega_{can},i_{can})=
\sum_{ 0 \ll \xi \in
\mathfrak{b}\theta^{-1}_{F'/F}}a(\xi,\phi',k))q^{\xi}
\]
with
\[
a(\xi,\phi',k)= \sum_{(a,b) \in (\mathfrak{r}' \times
\mathfrak{b}\theta^{-1}_{F'/F})/{\mathfrak{r}'}^{\times}, ab=\xi}
\phi'(a,b) sgn(N(a))N(a)^{k-1}
\]
As the function $\phi'$ is supported on the units of
${\rr'}_p^{\times}$ with respect to the second variable (i.e. the
$b$'s above) we have that the above $q$-expansion with respect the
selected cusp is given by
\[
E_{k}(\phi',\mathfrak{c}\theta_{F'/F})(Tate_{\rr',\mathfrak{b}\theta_{F'/F}^{-1}}(q),\lambda_{can},\omega_{can},i_{can})=
\sum_{ 0 \ll \xi \in \mathfrak{b}}a(\xi,\phi',k))q^{\xi}
\]
with
\[
a(\xi,\phi',k)= \sum_{(a,b) \in (\mathfrak{r}' \times
\mathfrak{b})/{\mathfrak{r}'}^{\times}, ab=\xi} \phi'(a,b)
sgn(N(a))N(a)^{k-1}
\]

From Lemma \ref{diagonalmapL} and the discussion after that it
follows that the $q$-expansion of the restricted Eisenstein
$res_{\Delta}E_k(\phi',\mathfrak{c}\theta_{F'/F})$ series at the
cusp $(\rr,\mathfrak{b})$ is given by
\[
res_{\Delta}E_{k}(\phi',\mathfrak{c}\theta_{F'/F})(Tate_{\rr,\mathfrak{b}}(q),\lambda_{can},\omega_{can},i_{can})=
\sum_{ 0 \ll \xi \in \mathfrak{b}}a(\xi,\phi',k)q^{\xi}
\]
where
\[
a(\xi,\phi',k)= \sum_{\xi' \in \mathfrak{b}, Tr_{F'/F}(\xi')=\xi}
a(\xi',\phi',k)
\]
The $q$-expansion of the Eisenstein series
$E_{pk}(\phi,\mathfrak{c})$ at the cusp $(\rr,\mathfrak{b})$ is
given by
\[
E_{pk}(\phi,\mathfrak{c})(Tate_{\rr,\mathfrak{b}}(q),\lambda_{can},\omega_{can},i_{can})=
\sum_{ 0 \ll \xi \in \mathfrak{b}}a(\xi,\phi,pk))q^{\xi}
\]
with
\[
a(\xi,\phi,pk)= \sum_{(a,b) \in (\rr \times \mathfrak{b}
)/{\mathfrak{r}}^{\times}, ab=\xi} \phi(a,b) sgn(N(a))N(a)^{pk-1}
\]
and hence that of $Frob_{p}(E_{pk}(\phi,\mathfrak{c}))$ is given by
\[
Frob_{p}(E_{pk}(\phi,\mathfrak{c})(Tate_{\rr,
\mathfrak{b}}(q),\lambda_{can},\omega_{can},i_{can})=\sum_{ 0 \ll
\xi \in \mathfrak{b}}a(\xi,\phi,pk))q^{p\xi}
\]
In order to establish the congruences of the Eisenstein series it is
enough, thanks to the $q$-expansion principle to establish the
congruences between the $q$-expansions at the selected cusp
$(\rr,\mathfrak{b})$.

We start by observing that the Eisenstein series
$Frob_{p}(E_{pk}(\phi,\mathfrak{c}))$ has non-zero terms only at
terms divisible by $p$ as we assume that the ideal $\mathfrak{b}$ is
prime to $p$. We consider the $\xi^{th}$-term of
$res_{\Delta}E_{k}(\phi',\mathfrak{c})$. It is equal to
\[
a(\xi,\phi',k)= \sum_{\xi' \in \mathfrak{b},
Tr_{F'/F}(\xi')=\xi}\sum_{(a,b) \in (\mathfrak{r}' \times
\mathfrak{b})/{\mathfrak{r}'}^{\times}, ab=\xi'} \phi'(a,b)
sgn(N(a))N(a)^{k-1}
\]
We observe that the group $\Gamma=Gal(F'/F)$ acts on the triples
$(\xi',a,b)$ of the summation above by $(\xi',a,b)^{\gamma}:=
({\xi'}^\gamma,a^\gamma,b^{\gamma})$ as $\mathfrak{b}$ is an ideal
of $F$ hence is preserved by $\Gamma$, where the action on $a$ and
$b$ is modulo the units in $\rr'$ to understand. We write $\gamma$
for a generator of $\Gamma$. We consider two cases, the case where
$(\xi,a,b)$ is fixed by $\gamma$ and the case where it is not. In
the first case we notice that as $\phi'$ is fixed under $\Gamma$ we
have that $\phi'(a^{\gamma},b^{\gamma})=\phi'(a,b)$. Hence we have
\[
\sum_{i=0}^{p-1}\phi'(a^{\gamma^{i}},b^{\gamma^{i}})sgn(N(a^{\gamma^{i}}))N(a^{\gamma^{i}})^{k-1}
= p \,\,\phi'(a,b) sgn(N(a))N(a)^{k-1} \equiv 0\,\,\, \mod{p}
\]
If $(\xi',a,b)$ is fixed by $\gamma$ then that implies that (i)
$\xi' \in F$ and (ii) the ideals generated by $a$ and $b$ in $\rr'$
are coming from ideals in $\rr$ as they are relative prime to
$\theta_{F'/F}$ i.e. to the primes where the extension is ramified.
Moreover as we assume that $Cl_{F} \hookrightarrow Cl_{F'}$ we have
that actually the elements themselves are (up to units) equal to
elements from $F$. In this case we first notice that
$\xi=Tr_{F'/F}(\xi')=p\xi'$ and as $\xi' \in
\mathfrak{b}\mathfrak{r}'$ with $\mathfrak{b}$ prime to $p$ we have
that $\xi$ is also divisible by $p$ in the sense that is of the form
$p\xi'$ for $\xi' \in \mathfrak{b}$. Further we have the congruences
modulo $p$
\[
\phi'(a,b) sgn(N_{F'}(a))N_{F'}(a)^{k-1} \equiv \phi(a,b)
sgn(N_F(a)^{p})N_{F}(a)^{p(k-1)}\]
\[
\equiv \phi(a,b) sgn(N_F(a))N_{F}(a)^{pk-1} \mod{p}
\]
In particular we conclude that $a(\xi,\phi',k) \equiv 0 \mod{p}$ if
$\xi$ is not of the form $p \xi'$ for $\xi' \in \mathfrak{b} \subset
F$. In the case where $\xi$ is of the form $p\xi'$ we have seen that
\[
a(p\xi',\phi',k) \equiv \sum_{(a,b) \in (\rr,\mathfrak{b})
/{\mathfrak{r}}^{\times}, ab=\xi'}\phi(a,b)
sgn(N_F(a))N_{F}(a)^{pk-1} = a(\xi',\phi,pk) \mod{p}
\]
But $a(\xi',\phi,pk)$ is the ${p \xi'}^{th}$ Fourier term of
$Frob_{p}(E_{pk}(\phi,\mathfrak{c}))$ which allow us to conclude the
proof of the proposition.

\end{proof}

\section{Using the theory of Complex Multiplication}

Before we prove our main theorem we need to make some preparation.
In this section we explain how we can use the theory of complex
multiplication to understand how Frobenious operates on values of
Eisenstein series of CM points. We recall that we consider the CM
types $(K_0,\Sigma_0)$ and its lift $(K,\Sigma)$. Moreover by our
setting we have that the reflex field for both of these CM types is
simply $(K_0,\Sigma_0)$. We first note that since we assume that $p$
is unramified in $F$ then the triples
$(X(\mathfrak{U}),\lambda(\mathfrak{U}),i(\mathfrak{U})))$ are
defined over the ring of integers of $W=W(\bar{\mathbb{F}}_p)$ (see
\cite{Hida1} page 69). We write $\Phi$ for the extension of the
Frobenious in $Gal(\Q_p^{nr}/\Q_p)$ to $W$. In this section we prove
the following proposition which is just a reformulation of what is
done in \cite{Katz1} (page 539) in the case of quadratic imaginary
fields.

\begin{proposition}(Reciprocity law on CM points)\label{reciprocityP} For every
fractional ideal $\mathfrak{U}$ of the CM field $K$ and $\phi$ a
$\Z_p$ valued locally constant function we have the reciprocity law
\[
Frob_{p}(E_{pk}(\phi,\mathfrak{c})(X(\mathfrak{U}),\lambda(\mathfrak{U}),i(\mathfrak{U}))
=
(E_{pk}(\phi,\mathfrak{c})(X(\mathfrak{U}),\lambda(\mathfrak{U}),i(\mathfrak{U})))^{\Phi}
\]
\end{proposition}
\begin{proof}
Let us write $\mathcal{R}$ for the ring of integers of $W$. As we
are assuming that $\phi$ is $\Z_p$ valued and we know from above
that the triple
$(X(\mathfrak{U}),\lambda(\mathfrak{U}),i(\mathfrak{U}))$ is defined
over $\mathcal{R}$ we have that the value of the Eisenstein series
is in $\mathcal{R}$. From the compatibility of $p$-adic modular
forms with ring extensions and the fact that the Eisenstein series
is defined over $\Z_p$ we have that
\[
(E_{pk}(\phi,\mathfrak{c})(X(\mathfrak{U}),\lambda(\mathfrak{U}),i(\mathfrak{U})))^{\Phi}=
(E_{pk}(\phi,\mathfrak{c})(X(\mathfrak{U}),\lambda(\mathfrak{U}),i(\mathfrak{U})
\otimes_{\mathcal{R},\Phi} \mathcal{R}))
\]
where the tensor product is with respect to the map $\Phi:
\mathcal{R} \rightarrow \mathcal{R}$, i.e. the base change of the
triple $(X(\mathfrak{U}),\lambda(\mathfrak{U}),i(\mathfrak{U}))$
with respect to the frobenious map. But then from the theory of
complex multiplication see \cite{Lang} (Lemma 3.1 in page 61 and
Theorem 3.4 in page 66), the fact that the reflex field of
$(K,\Sigma)$ is $(K_0,\Sigma_0)$ and that $p$ is ordinary we have
that
\[
(X(\mathfrak{U}),\lambda(\mathfrak{U}),i(\mathfrak{U}))
\otimes_{\mathcal{R},\Phi} \mathcal{R} \cong
(X'(\mathfrak{U}),\lambda'(\mathfrak{U}),i'(\mathfrak{U}))
\]
where $(X'(\mathfrak{U}),\lambda'(\mathfrak{U}),i'(\mathfrak{U}))$
is the quotient obtained by $X/H_{can}$ with
$H_{can}:=i(\theta_F\otimes \mathbf{\mu}_{p})$ as explained in Katz
\cite{Katz2} page 223. Moreover as in Katz we have that the Tate
HBAV
$(Tate'_{\mathfrak{a},\mathfrak{b}}(q),\lambda'_{can},i'_{can})$ is
obtained from
$(Tate_{\mathfrak{a},\mathfrak{b}}(q),\lambda_{can},i_{can})$ by the
map $q \mapsto q^p$ from which we conclude the proposition.
\end{proof}

\section{Complex and $p$-adic periods.}
In this section we study the various periods (archimedean and
$p$-adic) that appear in the interpolation properties of the
$KHT$-measure. We also consider the relative situation and we focus
especially in the case of interest with $(K_0,\Sigma_0) < (K,\Sigma)
< (K',\Sigma')$.

{\textbf{The periods of Katz:}} We start by recalling the periods
defined by Katz and then showing that in the case of the twisted
measure the periods used remain unchanged. We follow Katz (see
\cite{Katz2} page 268) and fix a nowhere vanishing differential over
$A:=\{a \in \bar{\Q}: incl(p)(a) \in D_p\}$
\[
\omega:Lie(X(\mathfrak{R})) \cong \theta_{F}^{-1} \otimes A
\]
Then for any fractional ideal $\mathfrak{U}$ of $K$ that is relative
prime to the place induced by $incl(p)$ we have an identification
$Lie(X((\mathfrak{U}))=Lie(X(\mathfrak{R}))$ and hence one may use
the very same $\omega$ to fix a nowhere differential of
$X(\mathfrak{U})$ by
\[
\omega(\mathfrak{U}):Lie(X((\mathfrak{U}))=Lie(X(\mathfrak{R}))
\cong \theta_{F}^{-1} \otimes A
\]
We use $incl(\infty):A \hookrightarrow \C$ to define the standard
complex nowhere vanishing differential
$\omega_{trans}(X(\mathfrak{U}))$ associated to the torus
$\C^{\Sigma}/\Sigma(\mathfrak{U})$. Then as in Katz (\cite{Katz2},
Lemma 5.1.45) we have an element
$\Omega^{Katz}_{K}=(\ldots,\Omega(\sigma),\ldots)) \in
(\C^{\times})^{\Sigma}$ such that for all fractional ideals
$\mathfrak{U}$ of $K$ relative prime to $p$ we have
\[
\omega(\mathfrak{U}) = \Omega^{Katz}_{K}
\omega_{trans}(\mathfrak{U})
\]
Of course the same considerations hold for $K_0$ and $K'$.
Especially for $K'$ we want to compute also the periods for the
twisted HBAV $X(\mathfrak{U} \otimes \xi)$. From the isomorphism
$X(\mathfrak{U}) \cong X(\mathfrak{U}_j \otimes \xi^{-1})$ we have
that we can pick the invariant differentials
$\omega(\mathfrak{U}\otimes \xi^{-1})$ and
$\omega_{trans}(\mathfrak{U}\otimes \xi^{-1})$ as $\xi \cdot
\omega(\mathfrak{U})$ and $\xi \cdot \omega_{trans}(\mathfrak{U})$
respectively. In particular we have that the selected periods are
equal to $\Omega^{Katz}_{K'}$. Similarly Katz (\cite{Katz2} Lemma
5.1.47) defines $p$-adic periods in $(D_p^{\times})^{\Sigma}$
relating the invariant differential $\omega(\mathfrak{U})$ to the
invariant differential $\omega_{can}(\mathfrak{U})$ obtained from
the $p^{\infty}$-structure. As above we obtain that the $p$-adic
periods for the twisted HBAV are the same.

{\textbf{Picking the periods compatible:}}(See also \cite{deShalit}
page 195 on the properties of the periods defined by Katz). Now we
consider the more specific setting where $(K,\Sigma)$ and
$(K',\Sigma')$ are lifted from the type $(K_0,\Sigma_0)$. Moreover
as we assume that $K_0$ is the CM field of an elliptic curve defined
over $\Q$, we have that $\mathfrak{R}_0$ has class number one, i.e.
it is a P.I.D. That means that the ring of integers $\mathfrak{R}$
and $\mathfrak{R}'$ are free over $\mathfrak{R}_0$. That means that
we have
\[
Lie(X(\mathfrak{R})) = \oplus_{j=1}^{g}Lie(X(\mathfrak{R}_0))
\]
and similarly
\[
Lie(X(\mathfrak{R}')) = \oplus_{j=1}^{g'}Lie(X(\mathfrak{R}_0))
\]
In particular that implies that
\[
\Omega_{K}^{Katz}= (\ldots,
\Omega(E),\ldots),\,\,\,and\,\,\,\Omega_{K'}^{Katz}=(\ldots,\Omega(E),\ldots)
\]
Similarly for the $p$-adic periods we observe that $X(\mathfrak{R})
\cong E \times \ldots \times E$ and hence
$X(\mathfrak{R})[p^\infty]\cong E[p^\infty] \times \ldots \times
E[p^{\infty}]$ where $E$ is the elliptic curve defined over $\Q$
that corresponds to the ideal $\mathfrak{R}_0$ with respect to the
CM type $(K_0,\Sigma_0)$. These considerations imply that
\[
\Omega_{p,K}^{Katz}=
(\ldots,\Omega_p(E),\ldots),\,\,\,and\,\,\,\Omega_{p,K'}^{Katz}=(\ldots,\Omega_p(E),\ldots)
\]

We note that the definition of the periods of Katz in general are
independent of the Gr\"{o}ssencharacter in general since they depend
only on its infinite type. This is why it is important to pick the
differentials $\omega(\mathfrak{R})$ and $\omega(\mathfrak{R}')$
properly. And actually in our setting we have a very natural choice
by considering the elliptic curve $E/\Q$ to whom the
Gr\"{o}ssencharacter $\psi_0$ is attached (recall that
$\psi_K=\psi_0 \circ N_{K/\Q}$ and $\psi_{K'}= \psi_{0} \circ
N_{K'/\Q}$.

\section{Congruences of measures}

We are now ready to prove our main theorem. We recall that this
amounts to proving the following

\begin{theorem}If (i) $Cl^{-}_K(\mathfrak{J}) \cong Cl^{-}_{K'}(\mathfrak{J})^{\Gamma}$ (ii) $Cl_F(1)\hookrightarrow Cl_{F'}(1)$ and (iii) $\theta_{F'/F}=(\xi)$
with $\xi \gg 0$ and $\xi =\zeta\bar{\zeta}$ for $\zeta \in K'$ then
we have the congruences
\[
\frac{\int_{G_K}\epsilon \circ ver \,\,\,\,
d\mu^{KHT}_{\psi_K,\delta}}{\Omega_{p}(E)^g} \equiv
\frac{\int_{G_{K'}}\epsilon\,\,\,\,\,
d\mu^{KHT,tw}_{\psi_{K'},\delta,\xi}}{\Omega_p(E)^{pg}} \mod{p\Z_p}
\]
for all $\epsilon$ locally constant $\Z_p$-valued functions on
$G_{K'}$ with $\epsilon^{\gamma}=\epsilon$ and belong to the
cyclotomic part of it, i.e. when it is written as a sum of finite
order characters it is of the form $\epsilon= \sum c_{\chi}\chi$
with $\chi^{\tau} = \chi$.
\end{theorem}

The strategy for proving the above theorem is as follows. By
definition we have that the twisted $KHT$-measure is given as
\[
\int_{G'}\phi(g) \mu_{\delta,\xi}^{KHT,tw}(g):= \sum_j
\int_{T}\tilde{\phi}_j dE_j := \sum_j
E_1(\phi_j,\mathfrak{c}_j)(X(\mathfrak{U}^{\xi}_j),\lambda^{\xi}_{\delta}(\mathfrak{U}_j\otimes
\theta_{F'/F}^{-1}),\imath^{\xi}(\mathfrak{U}_j \otimes
\theta_{F'/F}^{-1}))
\]
We consider the set of representatives $\{\mathfrak{U}_j\}$ of
$Cl^{-}_{K'}(\mathfrak{J})$. If we consider the map
\[
\rho : Cl^{-}_{K}(\mathfrak{J}) \rightarrow
Cl^{-}_{K'}(\mathfrak{J})^{\Gamma}
\]
We may pick representatives of $Im(\rho)$ to be fractional ideals
$\mathfrak{U}_j$ with the property
$\mathfrak{U}_j^{\gamma}=\mathfrak{U}_j$ for all $\gamma \in
\Gamma$. Moreover we may pick the other representatives of
$Cl^{-}_{K'}(\mathfrak{J})$ such that if $\mathfrak{U}_j$ is a
representative then if $\mathfrak{U}^{\gamma}_j$ is not in the same
equivalent class as $\mathfrak{U}_j$ then it is also a
representative (and this must hold for all $\gamma \in \Gamma$). We
may split the twisted measure as follows,
\[
\int_{G'}\phi(g) \mu_{\delta,\xi}^{KHT}(g)= \sum_{\mathfrak{U}_j \in
Im(\rho)}
E_1(\phi_j,\mathfrak{c}_j)(X(\mathfrak{U}^{\xi}_j),\lambda^{\xi}_{\delta}(\mathfrak{U}_j\otimes
\theta_{F'/F}^{-1}),\imath^{\xi}(\mathfrak{U}_j \otimes
\theta_{F'/F}^{-1}))\]
\[
+ \sum_{\mathfrak{U}_j \not \in Im(\rho)}
E_1(\phi_j,\mathfrak{c}_j)(X(\mathfrak{U}^{\xi}_j),\lambda^{\xi}_{\delta}(\mathfrak{U}_j\otimes
\theta_{F'/F}^{-1}),\imath^{\xi}(\mathfrak{U}_j \otimes
\theta_{F'/F}^{-1}))
\]
Our strategy is to compare the first summand (i.e those CM points
that are coming from $K$) with the $KHT$-measure of $K$ through the
diagonal embedding that we have worked above. For the other part we
will prove directly that under the assumptions of our theorem is in
$p \Z_p$. We start with the following proposition

\begin{proposition}Let $\mathfrak{U}_j$ be a fractional ideal of $K'$. Then for $\phi$ a locally constant function invariant under $\Gamma$ we have,
\[
E_{k}(\phi,\mathfrak{c}^{\gamma}_j)(X({\mathfrak{U}_j^{\gamma}}^{(\xi)},\lambda^{\xi}_{\delta}(\mathfrak{U}^{\gamma}_j\otimes
\theta_{F'/F}^{-1}),\imath^{\xi}(\mathfrak{U}^{\gamma}_j \otimes
\theta_{F'/F}^{-1})))=E_k(\phi_j,\mathfrak{c}_j)(X(\mathfrak{U}^{\xi}_j),\lambda^{\xi}_{\delta}(\mathfrak{U}_j\otimes
\theta_{F'/F}^{-1}),\imath^{\xi}(\mathfrak{U}_j \otimes
\theta_{F'/F}^{-1}))
\]
for $\gamma \in \Gamma$.
\end{proposition}
\begin{proof}
The first thing that we note is that for $\phi$ with $\phi^\gamma =
\phi$ the following equality holds
\[
E_{k}(\phi,\mathfrak{c}^{\gamma}_j)(X({\mathfrak{U}_j^{\gamma}}^{(\xi)},\lambda^{\xi}_{\delta}(\mathfrak{U}^{\gamma}_j\otimes
\theta_{F'/F}^{-1}),\imath^{\xi}(\mathfrak{U}^{\gamma}_j \otimes
\theta_{F'/F}^{-1})))=E_{k}(\phi,\mathfrak{c}^{\gamma}_j)(X({\mathfrak{U}_j^{\gamma}}^{(\xi^\gamma)},\lambda^{\xi}_{\delta}(\mathfrak{U}^{\gamma}_j\otimes
\theta_{F'/F}^{-1}),\imath^{\xi}(\mathfrak{U}^{\gamma}_j \otimes
\theta_{F'/F}^{-1})))
\]
for all $\gamma \in \Gamma$. Indeed it is enough to observe that
$\frac{\xi^{\gamma}}{\xi}\in \mathfrak{R}^{\times}$ and hence we
have the equality of ideals $\mathfrak{U}^{\gamma}_j \otimes (\xi) =
\mathfrak{U}^{\gamma}_j \otimes (\xi^{\gamma})$. We now have from
the definition of the Eisenstein series
\[
E_k(\phi,\mathfrak{c}_j)(X({\mathfrak{U}_j}^{(\xi)},\lambda^{\xi}_{\delta}(\mathfrak{U}_j\otimes
\theta_{F'/F}^{-1}),\imath^{\xi}(\mathfrak{U}_j \otimes
\theta_{F'/F}^{-1})))=\]
\[
\frac{(-1)^{kg'}\Gamma(k+s)^{g'}}{\sqrt{(D_{F'})}}\sum_{w \in
(\mathfrak{U}_j \otimes (\xi))(\ff
p)/\rr^{\times}}\frac{P\phi(w)}{N(w)^k|N(w)^{2s}|}\mid_{s=0}
\]
As we assume that $\phi^\gamma=\phi$ for all $\gamma \in \Gamma$ we
have that $P\phi(w^{\gamma})=P\phi(w)$. Indeed from the definition
of the partial Fourier transform we have
\[
P\phi(x,y)=p^{\alpha[F':\Q]N(\mathfrak{f})^{-1}} \sum_{a \in
X_{\alpha}}\phi(a,y)e_{F'}(ax)
\]
for $\phi$ factoring through $X_\alpha \times \mathfrak{r'}_p \times
(\mathfrak{r'}/\mathfrak{f})$ with $X_\alpha:=
\mathfrak{r'}_p/\alpha r'_p \times (\mathfrak{r'}/\mathfrak{f})$
with $\alpha \in \mathbb{N}$. But then
\[
P\phi(x^{\gamma},y^{\gamma})= p^{\alpha[F':\Q]N(\mathfrak{f})^{-1}}
\sum_{a \in X_{\alpha}}\phi(a,y^{\gamma})e_{F'}(ax^{\gamma})
\]
As $\gamma$ permutes $X_\alpha$ we have
\[
\sum_{a \in X_{\alpha}}\phi(a,y^{\gamma})e_{F'}(ax^{\gamma}) =
\sum_{a \in
X_{\alpha}}\phi(a^{\gamma},y^{\gamma})e_{F'}(a^{\gamma}x^{\gamma})=\sum_{a
\in X_{\alpha}}\phi(a,y)e_{F'}(ax)
\]
which concludes our claim.

Back to our considerations we have that
\[
\sum_{w \in (\mathfrak{U}_j \otimes (\xi))(\ff
p)/\rr^{\times}}\frac{P\phi(w)}{N(w)^k|N(w)^{2s}|}\mid_{s=0} =
\sum_{w \in (\mathfrak{U}_j \otimes (\xi))(\ff
p)/\rr^{\times}}\frac{P\phi(w^{\gamma})}{N(w^{\gamma})^k|N(w^{\gamma})^{2s}|}\mid_{s=0}
\]
But the last sum is equal to $\sum_{w \in (\mathfrak{U}^{\gamma}_j
\otimes (\xi^{\gamma}))(\ff
p)/\rr^{\times}}\frac{P\phi(w)}{N(w)^k|N(w)^{2s}|}\mid_{s=0}$ which
concludes the proof.
\end{proof}

We know consider the measure $\mu_{\psi',\delta,\xi}^{KHT}$. We
recall that $\psi'$ is a Gr\"{o}ssencharacter of type $1\Sigma$. We
write $\psi'_{finite}$ for its finite part. Then we define introduce
the notation for a locally constant function $\phi_j$,
\[
E_{\psi'}(\phi_j,\mathfrak{c}_j)(X(\mathfrak{U}^{\xi}_j),\lambda^{\xi}_{\delta}(\mathfrak{U}_j\otimes
\theta_{F'/F}^{-1}),\imath^{\xi}(\mathfrak{U}_j \otimes
\theta_{F'/F}^{-1})):=E_{1}(\phi_j\psi'_{finite,j},\mathfrak{c}_j)(X(\mathfrak{U}^{\xi}_j),\lambda^{\xi}_{\delta}(\mathfrak{U}_j\otimes
\theta_{F'/F}^{-1}),\imath^{\xi}(\mathfrak{U}_j \otimes
\theta_{F'/F}^{-1}))
\]
Moreover we define the subset $S$ of the selected representatives of
$Cl_{K'}^{-}(\mathfrak{J})$ as the set of ideals that represent
classes in $Cl_{K'}^{-}(\mathfrak{J})^{\Gamma}$ but not in
$Im(\rho)$.
\begin{corollary}For the twisted $KHT$-measure we have the
congruences
\[
\int_{G'}\phi(g) \mu_{\psi',\delta,\xi}^{KHT}(g)\equiv
\sum_{\mathfrak{U}_j \in Im(\rho)}
E_{\psi'}(\phi_j,\mathfrak{c}_j)(X(\mathfrak{U}^{\xi}_j),\lambda^{\xi}_{\delta}(\mathfrak{U}_j\otimes
\theta_{F'/F}^{-1}),\imath^{\xi}(\mathfrak{U}_j \otimes
\theta_{F'/F}^{-1}))\]
\[
+ \sum_{\mathfrak{U}_j \in S}
E_{\psi'}(\phi_j,\mathfrak{c}_j)(X(\mathfrak{U}^{\xi}_j),\lambda^{\xi}_{\delta}(\mathfrak{U}_j\otimes
\theta_{F'/F}^{-1}),\imath^{\xi}(\mathfrak{U}_j \otimes
\theta_{F'/F}^{-1})) \mod{p}
\]
for all $\Z_p$-valued locally constant functions $\phi$ of $G'$ such
that $\phi^{\gamma}=\phi$ for all $\gamma \in \Gamma$.
\end{corollary}
\begin{proof} It follows directly from the fact that $|\Gamma|=p$
and that $\phi^{\gamma}=\phi$ for all $\gamma \in \Gamma$.
\end{proof}
Our next aim is to prove the following proposition
\begin{proposition} Under our assumption, for all $\Z_p$-valued locally constant $\phi$ with $\phi^{\gamma}=\phi$ for all $\gamma \in \Gamma$, we have the congruences
\[
\Phi(\int_{G}(\phi \circ ver)(g) \mu_{\psi^{p},\delta}^{KHT}(g))
\equiv \sum_{\mathfrak{U}_j \in Im(\rho)}
E_{\psi'}(\phi_j,\mathfrak{c}_j)(X(\mathfrak{U}^{\xi}_j),\lambda^{\xi}_{\delta}(\mathfrak{U}_j\otimes
\theta_{F'/F}^{-1}),\imath^{\xi}(\mathfrak{U}_j \otimes
\theta_{F'/F}^{-1})) \mod{p}
\]
where $\Phi$ was the extension of the Frobenious element from its
action on $\Q^{nr}_p$ to its $p$-adic completion
$\mathcal{J}_\infty$.
\end{proposition}
\begin{proof}
By definition we have that
\[
\int_{G}(\phi\circ ver)(g) \mu_{\psi^p,\delta}^{KHT}(g) = \sum_j
E_{(\psi^p)}(\phi\circ
ver_j,\mathfrak{c}_j)(X(\mathfrak{U}_j),\lambda_{\delta}(\mathfrak{U}_j),\imath(\mathfrak{U}_j)
\]
where the sum runs over a set of representatives of
$Cl_K^{-}(\mathfrak{J})$ and
\[
E_{(\psi^p)}(\phi\circ
ver_j,\mathfrak{c}_j)(X(\mathfrak{U}_j),\lambda_{\delta}(\mathfrak{U}_j),\imath(\mathfrak{U}_j):=E_{p}(\phi
\psi_{finite}'\circ
ver_j,\mathfrak{c}_j)(X(\mathfrak{U}_j),\lambda_{\delta}(\mathfrak{U}_j),\imath(\mathfrak{U}_j)
\]
where we note that $\psi' \circ ver = \psi^p$ as $\psi'=\psi \circ
N_{K'/K}$. From the congruences between the Eisenstein series that
we have proved in Proposition \ref{CongruencesP} we have that
\[
Frob_p(E_{\psi^p}(\phi\circ
ver)_j,\mathfrak{c}_j)(X(\mathfrak{U}_j),\lambda_{\delta}(\mathfrak{U}_j),\imath(\mathfrak{U}_j))
\equiv
E_{\psi'}(\phi_j,\mathfrak{c}_j)(X(\mathfrak{U}^{\xi}_j),\lambda^{\xi}_{\delta}(\mathfrak{U}_j\otimes
\theta_{F'/F}^{-1}),\imath^{\xi}(\mathfrak{U}_j \otimes
\theta_{F'/F}^{-1}))
\]
where of course in the right hand side $\mathfrak{U}_j$ is
understood as $\mathfrak{U}_j\mathfrak{R}'$. We sum over all
representatives of $Cl_K^{-}(\mathfrak{J})$ and after using the Main
Theorem of Complex Multiplication and our assumption that $\rho$ is
injective we obtain
\[
\Phi(\int_{G}(\phi\circ ver)(g) \mu_{\psi^p,\delta}^{KHT}(g))\equiv
\]
\[\equiv
\sum_{\mathfrak{U}_j \in Im(\rho)}
E_{\psi'}(\phi_j,\mathfrak{c}_j)(X(\mathfrak{U}^{\xi}_j),\lambda^{\xi}_{\delta}(\mathfrak{U}_j\otimes
\theta_{F'/F}^{-1}),\imath^{\xi}(\mathfrak{U}_j \otimes
\theta_{F'/F}^{-1})) \mod{p}
\]
\end{proof}

\begin{lemma} Let $\phi$ be a locally constant $\Z_p$-valued function
of $G_K$ that is cyclotomic i.e. $\phi$ is the restriction to $G_K$
of a locally constant function on $G_F$. Then we have that
\[
\frac{\int_{G}\phi(g)
\mu_{\psi^k,\delta}^{KHT}(g)}{\Omega_p(E)^{gk}} \in \Z_p
\]
for all $k \in \mathbb{N}$
\end{lemma}
\begin{proof}This follows almost directly Lemma \ref{rationalityL} and the discussion after it. Indeed we
may write $\phi=\sum_{\chi}c_\chi \chi$ where $\chi$ are cyclotomic
i.e. $\chi \circ c = \chi$. For such characters it is known that for
all $\sigma \in Gal(\bar{\Q_p}/\Q_p)$ we have
\[
\left(\frac{\int_{G}\chi(g)
\mu_{\psi^k,\delta}^{KHT}(g)}{\Omega_p(E)^{gk}}\right)^{\sigma} =
\frac{\int_{G}(\chi(g))^{\sigma}
\mu_{\psi^k,\delta}^{KHT}(g)}{\Omega_p(E)^{gk}}
\]
For all $\sigma \in G_{\Q_p}$ and $\phi$'s cyclotomic we have
\[
\left(\frac{\int_{G}\phi(g)
\mu_{\psi^k,\delta}^{KHT}(g)}{\Omega_p(E)^{gk}}\right)^{\sigma}
=\sum_{\chi}c_{\chi}^{\sigma}\left(\frac{\int_{G}\chi(g)
\mu_{\psi^k,\delta}^{KHT}(g)}{\Omega_p(E)^{gk}}\right)^{\sigma}=
\sum_{\chi}c_{\chi}^{\sigma}\frac{\int_{G}(\chi(g))^{\sigma}
\mu_{\psi^k,\delta}^{KHT}(g)}{\Omega_p(E)^{gk}}
\]
But then as $\phi(g)=(\phi(g))^{\sigma} =
\sum_{\chi}c_{\chi}^{\sigma} \chi(g)^\sigma$ the last sum is equal
to $\frac{\int_{G}\phi(g)
\mu_{\psi^k,\delta}^{KHT}(g)}{\Omega_p(E)^{gk}}$ which finishes the
proof.
\end{proof}

Note that a direct corollary of the proposition is
\begin{corollary}If $\phi$ is cyclotomic then,
\[
\int_{G'}\phi(g) \mu_{\psi',\delta,\xi}^{KHT}(g) - u^g\int_{G}(\phi
\circ ver)(g) \mu_{\psi^{p},\delta}^{KHT}(g) \equiv
\]
\[
\equiv \sum_{\mathfrak{U}_j \in S}
E_{\psi'}(\phi_j,\mathfrak{c}_j)(X(\mathfrak{U}^{\xi}_j),\lambda^{\xi}_{\delta}(\mathfrak{U}_j\otimes
\theta_{F'/F}^{-1}),\imath^{\xi}(\mathfrak{U}_j \otimes
\theta_{F'/F}^{-1})) \mod{p}
\]
\end{corollary}
\begin{proof} We have
\[
\Phi\left(\int_{G}(\phi\circ ver)(g)
\mu_{\psi^p,\delta}^{KHT}(g)\right)=\Phi\left(\Omega_p(E)^{gp}\frac{\int_{G}(\phi\circ
ver)(g)
\mu_{\psi^p,\delta}^{KHT}(g)}{\Omega_p(E)^{gp}}\right)=u^{gp}\int_{G}(\phi\circ
ver)(g) \mu_{\psi^p,\delta}^{KHT}(g)
\]
as $\frac{\Omega_p(E)^{\Phi}}{\Omega_p(E)}=u$ and from the
assumption on $\phi$ we have that $\frac{\int_{G}(\phi\circ ver)(g)
\mu_{\psi^p,\delta}^{KHT}(g)}{\Omega_p(E)^{gp}} \in \Z_p$. But as
$u:=\psi_0(\bar{\pi}) \in \Z_p$ we have $u^p \equiv u \mod{p}$.
\end{proof}

\begin{lemma}We have the congruences
\[
u^g\frac{\int_{G}\phi(g)
\mu_{\psi^p,\delta}^{KHT}(g)}{\Omega_p(E)^{gp}} \equiv
\frac{\int_{G}\phi(g) \mu_{\psi,\delta}^{KHT}(g)}{\Omega_p(E)^{g}}
\mod{p}
\]
for all locally constant $\Z_p$-valued functions $\phi$ of $G$.
\end{lemma}
\begin{proof}As $\psi^p \equiv \psi \mod{p}$ we have that
\[
\int_{G}\phi(g) \mu_{\psi^p,\delta}^{KHT}(g) =
\int_{G}\phi(g)\psi^p\, \mu_{\delta}^{KHT}(g) \equiv
\int_{G}\phi(g)\psi\, \mu_{\delta}^{KHT}(g)= \int_{G}\phi(g)
\mu_{\psi,\delta}^{KHT}(g) \mod{p}
\]
Dividing by the unit $\Omega_p(E)^{pg}$ and observing that
$u=\frac{\Omega_p(E)^\Phi}{\Omega_p(E)}\equiv
\frac{\Omega_p(E)^p}{\Omega_p(E)} \mod{p}$ we have
\[
\frac{\int_{G}\phi(g)
\mu_{\psi^p,\delta}^{KHT}(g)}{\Omega_p(E)^{gp}} \equiv
\frac{\int_{G}\phi(g)
\mu_{\psi,\delta}^{KHT}(g)}{\Omega_p(E)^{g}}\times
\frac{\Omega_p(E)^g}{\Omega_p(E)^{pg}} \mod{p}
\]
which concludes the proof.
\end{proof}

Now our assumptions of the main theorem imply that $S = \emptyset$.
Then the last two statements conclude the proof of the main theorem.
Note that if we do not assume that $S=\emptyset$ then we obtain the
congruences
\[
\int_{G_F}\epsilon \circ ver \,\,\,\, d\mu_{E/F} \equiv
\int_{G_{F'}}\epsilon\,\,\,\,\, d\mu_{E/F'} + \Delta(\epsilon)
\mod{p\Z_p}
\]
where
\[
\Delta(\epsilon):=\frac{1}{\Omega_p(E)^{pg}}\sum_{\mathfrak{U}_j \in
S}
E_{\psi'}(\phi_j,\mathfrak{c}_j)(X(\mathfrak{U}^{\xi}_j),\lambda^{\xi}_{\delta}(\mathfrak{U}_j\otimes
\theta_{F'/F}^{-1}),\imath^{\xi}(\mathfrak{U}_j \otimes
\theta_{F'/F}^{-1}))
\]

\textbf{The Fukaya-Kato conjecture and the measure of Katz:} We
would like to finish this work by stating the question of whether
the $p$-adic interpolation properties of the Katz-Hida-Tilouine
measure are canonical. In \cite{FK} (page 67, theorem 4.2.22) Fukaya
and Kato conjecture a general formula for $p$-adic $L$ functions for
motives over any field. Does this formula agree with
Katz-Hida-Tilouine's formula in the case where the motive consider
is the one attached to a Gr\"{o}ssencharacter over a CM field? We
remark that our question is more concerning the $p$-adic and
archimedean periods that appear in the two formulas.

\textbf{Acknowledgements:} The author would like to thank Mahesh
Kakde for many helpful discussions. Moreover he would like to thank
Peter Barth and Jakob Stix for answering various questions of the
author. Finally the encouragement of Prof. John Coates and Prof.
Otmar Venjakob was very important to the author.

\section{Appendix}
There is an easy way to see that there must be a modification in the
interpolation properties of the measures in order for the
congruences to hold. We assume for simplicity that $F'/F$ ramifies
only above $p$. Moreover we assume that the character $\psi_{K}$ is
unramified (we just divide out the finite part of it which has
conductor $\mathfrak{f}$) and we pick with notation as in the
introductio $\mathfrak{n}=\mathfrak{r}$. Let us pick as the locally
constant function $\epsilon$ that appear in the congruences the
character $\tilde{\phi}:=\phi \circ N_{K'/K}$ for some finite
$\Z_p^\times$-valued character of $G_K$, which we assume cyclotomic
(for example $\phi:=\mathbf{1}$ or some of the $p-1$ order
characters factorizing through the torsion of $G_F$ base changed to
$G_K$). Then by the interpolation properties of the measure we have
\[
\frac{\int_{G_{K'}}\tilde{\phi}\,\,\,\,\,
d\mu_{\psi_{K'}}}{\Omega_p(E)^{pg}}= \prod_{\mathfrak{p}\in
\Sigma'_p}Local_{\mathfrak{p}}(\tilde{\phi}\psi_{K'},\Sigma',\delta')(1-\tilde{\phi}\check{\psi}_{K'}(\bar{\mathfrak{p}}))(1-\tilde{\phi}\psi_{K'}(\bar{\mathfrak{p}}))\frac{L(0,\tilde{\phi}\psi_{K'})}{\sqrt{|D_{F'}|}\Omega(E)^{pg}}=
\]
\[
\frac{\prod_{\mathfrak{p}\in
\Sigma'_p}Local_{\mathfrak{p}}(\tilde{\phi}\psi_{K'},\Sigma',\delta')}{\sqrt{|D_{F'}|}}\prod_{\chi}\prod_{\mathfrak{p}\in
\Sigma_p}(1-\phi\check{\psi}_{K}\chi(\bar{\mathfrak{p}}))(1-\phi\psi_{K}\chi(\bar{\mathfrak{p}}))\frac{L(0,\phi\psi_{K}\chi)}{\Omega(E)^{g}}
\]
where $\chi$ runs over the characters of the extension $K'/K$. Now
we note that $\chi \equiv 1 \mod{(\zeta_p-1)}$ and hence as
$Gal(K'/K)$ is a quotient of $G_K$ we have that
\[
\frac{\int_{G_K}\phi\chi \,\,\,\,
d\mu_{\psi_K}}{\Omega_{p}(E)^g}\equiv \frac{\int_{G_K} \phi \,\,\,\,
d\mu_{\psi_K}}{\Omega_{p}(E)^g} \mod{(\zeta_p -1) }
\]
or equivalently
\[
\prod_{\mathfrak{p} \in
\Sigma_p}Local_{\mathfrak{p}}(\phi\chi\psi_K,\Sigma,\delta)\prod_{\mathfrak{p}\in
\Sigma_p}(1-\phi\check{\psi}_{K}\chi(\bar{\mathfrak{p}}))(1-\phi\psi_{K}\chi(\bar{\mathfrak{p}}))\frac{L(0,\phi\psi_{K}\chi)}{\sqrt{|D_F|}\Omega(E)^{g}}\equiv
\]
\[
\equiv \prod_{\mathfrak{p}\in
\Sigma_p}Local_{\mathfrak{p}}(\phi\psi_K,\Sigma,\delta)(1-\phi\check{\psi}_{K}(\bar{\mathfrak{p}}))(1-\phi\psi_{K}(\bar{\mathfrak{p}}))\frac{L(0,\phi\psi_{K})}{\sqrt{|D_F|}\Omega(E)^{g}}
\mod{(\zeta_p -1)}
\]
Taking the product over all $\chi$'s we obtain
\[
\frac{\sqrt{|D_{F'}|}}{\sqrt{|D_F|^p}}\frac{\prod_{\chi}(\prod_{\mathfrak{p}
\in
\Sigma_p}Local_{\mathfrak{p}}(\phi\chi\psi_K,\Sigma,\delta))}{\prod_{\mathfrak{p}\in
\Sigma'_p}Local_{\mathfrak{p}}(\tilde{\phi}\psi_{K'},\Sigma',\delta')}\frac{\int_{G_{K'}}\tilde{\phi}\,\,\,\,\,
d\mu_{\psi_{K'}}}{\Omega_p(E)^{pg}} \equiv \left(\frac{\int_{G_K}
\phi \,\,\,\, d\mu_{\psi_K}}{\Omega_{p}(E)^g}\right)^p
\mod{(\zeta_p-1)}
\]
Now we note that
\[
\left(\frac{\int_{G_K} \phi \,\,\,\,
d\mu_{\psi_K}}{\Omega_{p}(E)^g}\right)^p \equiv \frac{\int_{G_K}
\phi^p \,\,\,\, d\mu_{\psi_K}}{\Omega_{p}(E)^g} \mod{p}
\]
as the values of the integrals are in $\Z_p$ as we assume that
$\phi$ is cyclotomic. Hence we need to understand the factor
$\frac{\sqrt{|D_{F'}|}}{\sqrt{|D_F|^p}}\frac{\prod_{\chi}(\prod_{\mathfrak{p}
\in
\Sigma_p}Local_{\mathfrak{p}}(\phi\chi\psi_K,\Sigma,\delta))}{\prod_{\mathfrak{p}\in
\Sigma'_p}Local_{\mathfrak{p}}(\tilde{\phi}\psi_{K'},\Sigma',\delta')}$
and where the quantity $\frac{\int_{G_{K'}}\tilde{\phi}\,\,\,\,\,
d\mu_{\psi_{K'}}}{\Omega_p(E)^{pg}}$ lies. We start with the local
factors. From Lemma \ref{epsilonfactors} we have that
\[
Local(\phi\chi\psi_K,\Sigma,\delta)_{\mathfrak{p}}=c^{(\chi)}_{\mathfrak{p}}(\delta)e_{\mathfrak{p}}(\phi^{-1}\chi^{-1},\psi,dx_1)\left(\frac{\psi^{-1}_K(\pi_{\mathfrak{p}})}{N(\mathfrak{p})}\right)^{n_{\p}(\phi\chi)+n_{\p}(\psi)}
\]
and
\[
Local(\tilde{\phi}\psi_K,\Sigma',\delta')_{\mathfrak{p}}=c'_{\mathfrak{p}}(\delta')e_{\mathfrak{p}}(\tilde{\phi}^{-1},\psi',dx_1)\left(\frac{\psi^{-1}_{K'}(\pi_{\mathfrak{p}})}{N(\mathfrak{p})}\right)^{n_{\p}(\tilde{\phi})+n_{\p}(\psi')}
\]
where $c^{(\chi)}_{\mathfrak{p}}$ the local part of $\phi\chi\psi_K$
and $dx_1$ is the Haar measure that assigns measure 1 to the ring of
integers of $K_{\mathfrak{p}}$ (with similar notations for the
second expression). Now we note that (as easily seen from the
functional equation and the fact that
$Ind^{K'}_{K}\mathbf{1}=\oplus_{\chi}\chi$) we have that
\[
\prod_{\mathfrak{p} \in
\Sigma'}e_{\mathfrak{p}}(\tilde{\phi},\psi',dx_{\psi}')=
\prod_{\chi}\prod_{\mathfrak{p} \in
\Sigma}e_{\mathfrak{p}}(\phi\chi,\psi,dx_{\psi})
\]
where we follow Tate's notation as in \cite{Tate} for the Tamagawa
measures $dx_{\psi}$ and $dx_{\psi'}$. The relation between the
Tamagawa measure $dx_{\psi}$ and the normalized measure $dx_1$ of a
place $\p$ is given by $dx_{\psi}=N(\p)^{-n_{\p}(\psi)/2} dx_1$
(There is a typo in Tate's \cite{Tate} p.17, but see the same
article in page 18 or Lang's Algebraic Number Theory page 277). That
implies,
\[
\prod_{\mathfrak{p} \in
\Sigma'}e_{\mathfrak{p}}(\tilde{\phi},\psi',dx_{\psi}')=\prod_{\mathfrak{p}
\in
\Sigma'}e_{\mathfrak{p}}(\tilde{\phi},\psi',dx_{1})N(\mathfrak{p})^{-n_{\p}(\psi')/2}
\]
and
\[
\prod_{\chi}\prod_{\mathfrak{p} \in
\Sigma}e_{\mathfrak{p}}(\phi\chi,\psi,dx_{\psi})=\prod_{\chi}\prod_{\mathfrak{p}
\in
\Sigma}e_{\mathfrak{p}}(\phi\chi,\psi,dx_{1})N(\mathfrak{p})^{-n_{\p}(\psi)/2}=\prod_{\mathfrak{p}
\in
\Sigma}N(\mathfrak{p})^{-pn_{\p}(\psi)/2}\prod_{\chi}e_{\mathfrak{p}}(\phi\chi,\psi,dx_{1})
\]
So we conclude the equation
\[
\prod_{\mathfrak{p} \in
\Sigma'}e_{\mathfrak{p}}(\tilde{\phi},\psi',dx_{1})N(\mathfrak{p})^{-n_{\p}(\psi')/2}=\prod_{\mathfrak{p}
\in
\Sigma}N(\mathfrak{p})^{-pn_{\p}(\psi)/2}\prod_{\chi}e_{\mathfrak{p}}(\phi\chi,\psi,dx_{1})
\]
or equivalently
\[
\prod_{\chi}\prod_{\mathfrak{p} \in
\Sigma}e_{\mathfrak{p}}(\phi\chi,\psi,dx_{1})=\frac{\prod_{\mathfrak{p}
\in \Sigma'}N(\mathfrak{p})^{-n_{\p}(\psi')/2}}{\prod_{\mathfrak{p}
\in \Sigma}N(\mathfrak{p})^{-pn_{\p}(\psi)/2}}\prod_{\mathfrak{p}
\in \Sigma'}e_{\mathfrak{p}}(\tilde{\phi},\psi',dx_{1})
\]
As we assume that $\Sigma$ and $\Sigma'$ are ordinary and for
simplicity we take the extension to be ramified only at $p$ we have
that $\frac{\prod_{\mathfrak{p} \in
\Sigma'}N(\mathfrak{p})^{n_{\p}(\psi')/2}}{\prod_{\mathfrak{p} \in
\Sigma}N(\mathfrak{p})^{pn_{\p}(\psi)/2}}=\frac{\sqrt{|D_{F'}|}}{\sqrt{|D_F|^p}}$.
Putting everything together we see that the discrepancy factor in
the congruences
\[
Diff:=\frac{\sqrt{|D_{F'}|}}{\sqrt{|D_F|^p}}\times
\frac{\prod_{\chi}(\prod_{\mathfrak{p} \in
\Sigma_p}Local_{\mathfrak{p}}(\phi\chi\psi_K,\Sigma,\delta))}{\prod_{\mathfrak{p}\in
\Sigma'_p}Local_{\mathfrak{p}}(\tilde{\phi}\psi_{K'},\Sigma',\delta')}
\]
is equal to
\[
Diff= \frac{\prod_{\chi}\prod_{\mathfrak{p} \in
\Sigma_p}c^{(\chi)}_{\mathfrak{p}}(\delta)\left(\frac{\psi^{-1}_K(\pi_{\mathfrak{p}})}{N(\mathfrak{p})}\right)^{n_{\p}(\phi\chi)+n_{\p}(\psi)}}{\prod_{\p
\in
\Sigma'_p}c'_{\p}(\delta')\left(\frac{\psi^{-1}_{K'}(\pi_{\mathfrak{p}})}{N(\mathfrak{p})}\right)^{n_{\p}(\tilde{\phi})+n_{\p}(\psi')}}
\]
Now we claim that the factor
\[
\frac{\prod_{\chi}\prod_{\mathfrak{p} \in
\Sigma_p}\left(\frac{\psi^{-1}_K(\pi_{\mathfrak{p}})}{N(\mathfrak{p})}\right)^{n_{\p}(\phi\chi)+n_{\p}(\psi)}}{\prod_{\p
\in
\Sigma'_p}\left(\frac{\psi^{-1}_{K'}(\pi_{\mathfrak{p}})}{N(\mathfrak{p})}\right)^{n_{\p}(\tilde{\phi})+n_{\p}(\psi')}}=1.
\]
Indeed we have
\[
\prod_{\p \in
\Sigma'_p}\left(\frac{\psi^{-1}_{K'}(\pi_{\mathfrak{p}})}{N(\mathfrak{p})}\right)^{n_{\p}(\tilde{\phi})+n_{\p}(\psi')}=\prod_{\p
\in \Sigma'_p}\left(\frac{\psi^{-1}_{K}\circ
N_{K'/K}(\pi_{\mathfrak{p}})}{N_K\circ
N_{K'/K}(\mathfrak{p})}\right)^{n_{\p}(\tilde{\phi})+n_{\p}(\psi')}
\]
For those $\p' \in \Sigma'_p$ that are not ramified we have
$n_{\p'}(\psi')=n_{\p}(\psi)$ for $\p \in \Sigma_p$ the prime below
$\p'$. Similarly
$n_{\p'}(\tilde{\phi})=n_{\p}(\phi\chi)=n_{\p}(\phi)$ for all the
$\chi$ as these are ramified only at the primes that ramify in
$K'/K$. Then we have
\[
\prod_{\p \in \Sigma'_p,\,\, unram.}\left(\frac{\psi^{-1}_{K}\circ
N_{K'/K}(\pi_{\mathfrak{p}})}{N_K\circ
N_{K'/K}(\mathfrak{p})}\right)^{n_{\p}(\tilde{\phi})+n_{\p}(\psi')}=\prod_{\p
\in \Sigma_p,
\,\,unram.}\left(\frac{\psi^{-1}_{K}(\pi_{\mathfrak{p}})}{N(\mathfrak{p})}\right)^{p(n_{\p}(\phi)+n_{\p}(\psi))}=
\]
\[
=\prod_{\chi}\prod_{\p \in \Sigma_p,
\,\,unram.}\left(\frac{\psi^{-1}_{K}(\pi_{\mathfrak{p}})}{N(\mathfrak{p})}\right)^{n_{\p}(\chi\phi)+n_{\p}(\psi)}
\]
Now we consider the ramified primes. We have
\[
\prod_{\p \in \Sigma'_p,\,\, ram.}\left(\frac{\psi^{-1}_{K}\circ
N_{K'/K}(\pi_{\mathfrak{p}})}{N_K\circ
N_{K'/K}(\mathfrak{p})}\right)^{n_{\p}(\tilde{\phi})+n_{\p}(\psi')}=\prod_{\p
\in
\Sigma_p}\left(\frac{\psi^{-1}_{K}(\pi_{\mathfrak{p}})}{N(\mathfrak{p})}\right)^{n_{\p}(\tilde{\phi})+n_{\p}(\psi')}
\]
For every $\p' \in \Sigma'_p$ that is ramified (totally as we
consider a $p$-order extension) we have from the
conductor-discriminant formula that
\[
n_{\p'}(\psi')= \sum_{\chi} n_{\p}(\chi) + pn_{\p}(\psi)
\]
for the prime $\p \in \Sigma_p$ below $\p$. Moreover as the
conductor-function $n_{\p}(\cdot)$ is additive and inductive in
degree zero we have that
\[
n_{\p'}(\tilde{\phi})=n_{\p'}(Res(\phi))=n_{\p'}(Res(\phi))-n_{\p'}(\mathbf{1})=n_{\p'}(Res(\phi)\ominus
\mathbf{1})=n_{\p}(Ind(Res(\phi))\ominus Ind(\mathbf{1}))=\]
\[
=n_{\p}(IndRes(\phi))-n_{\p}(Ind(\mathbf{1}))=n_{\p}(\oplus_{\chi}\phi\chi)-n_{\p}(\oplus_{\chi}\chi)=\sum_{\chi}n_{\p}(\phi\chi)-\sum_{\chi}n_{\p}(\chi)
\]
Putting all together we conclude our claim. Hence we have that
\[
Diff=\frac{\prod_{\chi}\prod_{\mathfrak{p} \in
\Sigma_p}c^{(\chi)}_{\mathfrak{p}}(\delta)\left(\frac{\psi^{-1}_K(\pi_{\mathfrak{p}})}{N(\mathfrak{p})}\right)^{n_{\p}(\phi\chi)+n_{\p}(\psi)}}{\prod_{\p
\in
\Sigma'_p}c'_{\p}(\delta')\left(\frac{\psi^{-1}_{K'}(\pi_{\mathfrak{p}})}{N(\mathfrak{p})}\right)^{n_{\p}(\tilde{\phi})+n_{\p}(\psi')}}
=\frac{\prod_{\chi}\prod_{\mathfrak{p} \in
\Sigma_p}c^{(\chi)}_{\mathfrak{p}}(\delta)}{\prod_{\p \in
\Sigma'_p}c'_{\p}(\delta')}
\]
Now we observe that
\[
\prod_{\chi}\prod_{\mathfrak{p} \in
\Sigma_p}c^{(\chi)}_{\mathfrak{p}}(\delta)=
\prod_{\chi}\prod_{\mathfrak{p} \in
\Sigma_p}(\phi\chi\psi_{K})_{\mathfrak{p}}(\delta)=\prod_{\mathfrak{p}
\in
\Sigma_p}(\phi\psi_K)_{\p}(\delta^p)\prod_{\chi}\chi_{\p}(\delta)\]
\[
=\prod_{\mathfrak{p} \in
\Sigma_p}(\phi\psi_K)_{\p}(\delta)\prod_{\chi}\chi_{\p}(\delta)=\prod_{\mathfrak{p}
\in \Sigma_p}(\phi\psi_K)_{\p}(\delta^p)
\]
since $\prod_{\chi}\chi_{\p}(\delta)=1$ because we multiply over all
elements of the multiplicative group of characters of $Gal(K'/K)$
and we know that $\chi \neq \chi^{-1}$ for all $\chi \neq 1$ as
these are $p$-order characters. Also we have that
\[
\prod_{\p \in \Sigma'_p}c'_{\p}(\delta')=\prod_{\p \in
\Sigma'_p}(\phi \circ N_{K'/K})_{\p}(\psi_K \circ
N_{K'/K})_{\p}(\delta')=\prod_{\p \in
\Sigma_p}(\phi\psi_K)_{\p}(N_{K'/K}\delta')
\]
In particular we observe that in general we have that
\[
\prod_{\chi}\prod_{\mathfrak{p} \in
\Sigma_p}c^{(\chi)}_{\mathfrak{p}}(\delta) \neq \prod_{\p \in
\Sigma'_p}c'_{\p}(\delta').
\]
as $N_{K'/K}(\delta') \neq \delta^p$ when the extension $K'/K$ is
ramified at $p$. Actually the two expressions may not even have the
same valuation.

\end{document}